%% file: Forcing_theorem.tex
\documentclass{amsart}

\title[]{The exact strength of the class forcing theorem}

\author[Gitman]{Victoria Gitman}
\address[Victoria Gitman]{The City University of New York, CUNY Graduate Center, Mathematics Program, 365 Fifth Avenue, New York, NY 10016, USA}
\email{vgitman@nylogic.org}
\urladdr{http://boolesrings.org/victoriagitman}

 \author{Joel David Hamkins}
 \address[Joel David Hamkins]
          {Professor of Logic, University of Oxford, and
           Sir Peter Strawson Fellow, University College, High Street, Oxford OX1 4BH, UK}
 \email{joeldavid.hamkins@philosophy.ox.ac.uk}
 \urladdr{http://jdh.hamkins.org}

\author[Holy]{Peter Holy}
\address[Peter Holy]{Mathematisches Institut, Universit\"at Bonn,
Endenicher Allee 60, 53115 Bonn, Germany}
\email{pholy@math.uni-bonn.de}
\urladdr{}

\author[Schlicht]{Philipp Schlicht}
\address[Philipp Schlicht]{School of Mathematics, University of Bristol, Fry Building, Woodland Road, Bristol, BS8 1UG, UK}
\email{philipp.schlicht@bristol.ac.uk}
\urladdr{}

\DeclareRobustCommand{\okina}{%
  \raisebox{\dimexpr\fontcharht\font`A-\height}{%
    \scalebox{0.8}{`}%
  }%
}
\author{Kameryn J. Williams}
\address[Kameryn J. Williams]{
University of Hawai\okina{}i at M\=anoa \\
Department of Mathematics \\
2565 McCarthy Mall, Keller 401A \\
Honolulu, HI  96822 \\
USA}
\email{kamerynw@hawaii.edu}
\urladdr{http://kamerynjw.net}

\thanks{The research of the second author has been supported by grant \#69573-00 47 from the CUNY Research Foundation. He is also grateful for the support he has received from the Hausdorff Center at the University of Bonn for several visits there. The third and fourth author were partially supported by DFG-grant LU2020/1-1. This project has received funding from the European Union's Horizon 2020 research and innovation programme under the Marie Sk\l odowska-Curie grant agreement No 794020 (IMIC) for the fourth author. Commentary concerning this paper can be made at http://jdh.hamkins.org/class-forcing-theorem.}
\include{MathMacrosJDH}
\usepackage[shortlabels]{enumitem} % for changing the enumeration labels
\usepackage{bbm}
\usepackage{needspace}
\usepackage{latexsym,amsfonts,amsmath,amssymb}
\usepackage{relsize}
\usepackage[rgb,dvipsnames]{xcolor}
\usepackage{tikz} % for diagrams and drawing
\usetikzlibrary{decorations.text}
\usetikzlibrary{arrows,arrows.meta,hobby}
\usetikzlibrary{shapes.misc, positioning}
\usepackage{diagrams}\diagramstyle[tight,centredisplay,textflow]
\usepackage[hidelinks]{hyperref}
\newcommand\ETRord{\ETR_{\Ord}}
\renewcommand\T{\mathord{\rm T}}
\newcommand\Tr{\mathord{\rm Tr}}
\newcommand\echeck{{\lower 3pt\hbox{$\check{}$}}}
\newcommand\FA{\F_{\!A}}
\newcommand\op{\operatorname{op}}

\keywords{Class forcing, forcing theorem, truth predicates, infinitary logic} 

\begin{document}

\begin{abstract}
The class forcing theorem, which asserts that every class forcing notion $\P$ admits a forcing relation $\forces_\P$, that is, a relation satisfying the forcing relation recursion---it follows that statements true in the corresponding forcing extensions are forced and forced statements are true---is equivalent over \Godel-Bernays set theory \GBC\ to the principle of elementary transfinite recursion $\ETRord$ for class recursions of length $\Ord$. It is also equivalent to the existence of truth predicates for the infinitary languages $\mathcal{L}_{\Ord,\omega}(\in,A)$, allowing any class parameter $A$; to the existence of truth predicates for the language $\mathcal{L}_{\Ord,\Ord}(\in,A)$; to the existence of $\Ord$-iterated truth predicates for first-order set theory $\mathcal{L}_{\omega,\omega}(\in,A)$; to the assertion that every separative class partial order $\P$ has a set-complete class Boolean completion; to a class-join separation principle; and to the principle of determinacy for clopen class games of rank at most $\Ord+1$. Unlike set forcing, if every class forcing notion $\P$ has a forcing relation merely for atomic formulas, then every such $\P$ has a uniform forcing relation applicable simultaneously to all formulas. Our results situate the class forcing theorem in the rich hierarchy of theories between $\GBC$ and Kelley-Morse set theory $\KM$.
\end{abstract}

\maketitle

\section{Introduction}

%\noindent
We shall characterize the exact strength of the class forcing theorem, which asserts that every class forcing notion $\P$ has a corresponding forcing relation $\forces_\P$, a relation satisfying the relevant forcing relation recursion. When there is such a forcing relation, then statements true in any corresponding forcing extension are forced and forced statements are true in those extensions.

Unlike set forcing, for which one may prove in \ZFC\ that every set forcing notion has corresponding forcing relations, with class forcing it is consistent with \Godel-Bernays set theory \GBC\ that there is a proper class forcing notion lacking a corresponding forcing relation, even merely for the atomic formulas. For certain forcing notions (see~\cite{HolyKrapfLuckeNjegomirSchlicht2016:Class-forcing, Krapf2017:Dissertation}, also theorem~\ref{Theorem.Atomic-forcing-relation-to-truth-predicate}), the existence of an atomic forcing relation implies $\Con(\ZFC)$ and much more, and so the consistency strength of the class forcing theorem strictly exceeds \GBC, if this theory is consistent. Nevertheless, the class forcing theorem is provable in stronger theories, such as Kelley-Morse set theory. What is the exact strength of the class forcing theorem?

Our project here is to identify the strength of the class forcing theorem by situating it in the rich hierarchy of theories between \GBC\ and \KM, displayed in part in figure~\ref{Figure.Theories}, with the class forcing theorem highlighted in blue. (The theory KM + Class Choice that appears at the top of figure~\ref{Figure.Theories} is defined e.g. in \cite[definition 2.11]{williams-min-km}.) It turns out that the class forcing theorem is equivalent over \GBC\ to an attractive collection of several other natural set-theoretic assertions; it is a robust axiomatic principle.
\begin{figure}[h]\label{Figure.Theories}
  \begin{tikzpicture}[theory/.style={draw,rounded rectangle,scale=.5,minimum height=6.5mm},scale=.3]
     % draw the nodes
     \draw (0:0) node[theory] (ZFC) {$\GBC\ \equiv\ \GB\ \equiv\ \ZFC\ \equiv\ \ZF$}
           ++(90:2) node[theory] (ConZFC) {$\GBC+\Con(\GBC)\quad\equiv\quad\ZFC+\Con(\ZFC)$}
           ++(90:2) node[theory] (ConalphaZFC) {$\GBC+\Con^\alpha(\GBC)\quad\equiv\quad\ZFC+\Con^\alpha(\ZFC)$}
           ++(90:2) node[theory] (ETRomega) {$\GBC+\ETR_\omega$}
           ++(90:2) node[theory] (ETRalpha) {$\GBC+\ETR_\alpha\quad=\quad\GBC+\alpha\text{-iterated truth predicates}$}
           ++(90:2) node[theory] (ETR<Ord) {$\GBC+\ETR_{<\Ord}\quad=\quad \GBC+\forall\alpha\exists \alpha\text{-iterated truth predicates}$}
           ++(90:5) node[theory,very thick,blue!80!black,scale=1.2] (ETRord) {\parbox[c][2.7cm]{9cm}{{\quad}$\GBC+\text{ Class forcing theorem}\quad=\quad \GBC+\ETRord$\\
            ${\qquad}=\quad\GBC+\text{ truth predicates for }\mathcal{L}_{\Ord,\omega}(\in,A)$\\
            ${\qquad}=\quad\GBC+\text{ truth predicates for }\mathcal{L}_{\Ord,\Ord}(\in,A)$\\
            ${\qquad}=\quad\GBC+\text{ $\Ord$-iterated truth predicates }$\\
            ${\qquad}=\quad\GBC\,+$ Boolean set-completions exist\quad $=$ \\
            $\GBC+\text{Determinacy for clopen class games of rank }\Ord+1$}}
           ++(90:5) node[theory] (ETROrd*omega) {$\GBC+\ETR_{\Ord\cdot\omega}$}
           ++(90:2) node[theory] (ETR) {$\GBC+\ETR\quad=\quad\GBC+\text{ Determinacy for clopen class games}$}
           ++(90:2) node[theory] (Open) {$\GBC+\text{ Determinacy for open class games}$}
           ++(90:2) node[theory] (Pi11) {$\GBC+\Pi^1_1$-comprehension}
           ++(90:2) node[theory] (KM) {$\KM\quad \equiv\quad\KM+\text{Class Choice}$};%_{\phantom{+}}$};
     % draw the arrows
     \draw[<-,>=stealth,thick]
                        (ZFC) edge (ConZFC)
                        (ConZFC) edge (ConalphaZFC)
                        (ConalphaZFC) edge (ETRomega)
                        (ETRomega) edge (ETRalpha)
                        (ETRalpha) edge (ETR<Ord)
                        (ETR<Ord) edge (ETRord)
                        (ETRord) edge (ETROrd*omega)
                        (ETROrd*omega) edge (ETR)
                        (ETR) edge (Open)
                        (Open) edge (Pi11)
                        (Pi11) edge (KM);
  \end{tikzpicture}
  \caption{A hierarchy of theories between $\GBC$ and $\KM$ by consistency strength ($\equiv$ means equiconsistent)}
\end{figure}
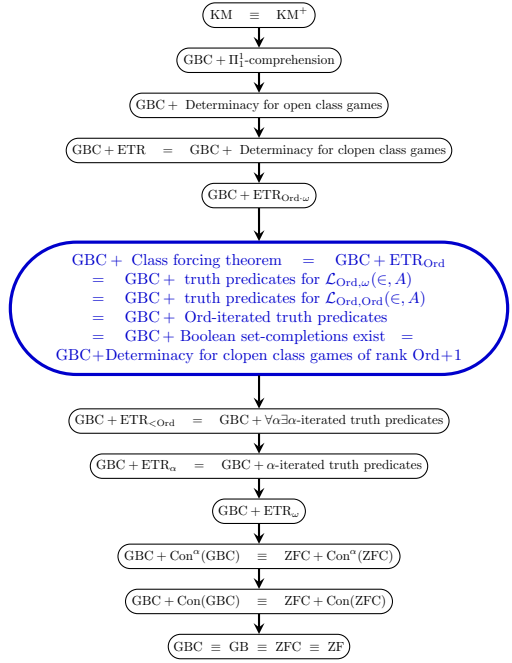

The main theorem is naturally part of the emerging subject we call the reverse mathematics of second-order set theory, a higher analogue of the perhaps more familiar reverse mathematics of second-order arithmetic. In this new research area, we are concerned with the hierarchy of second-order set theories between \GBC\ and \KM\ and beyond, analyzing the strength of various assertions in second-order set theory, such as the principle \ETR\ of elementary transfinite recursion, the principle of $\Pi^1_1$-comprehension or the principle of determinacy for clopen class games. We fit these set-theoretic principles into the hierarchy of theories over the base theory \GBC. The main theorem of this article does exactly this with the class forcing theorem by finding its exact strength in relation to nearby theories in this hierarchy.

Specifically, extending the analysis of~\cite{HolyKrapfLuckeNjegomirSchlicht2016:Class-forcing, HolyKrapfSchlicht2017:Characterizations-of-pretameness, Krapf2017:Dissertation, GitmanHamkins2016:OpenDeterminacyForClassGames}, we show that the class forcing theorem is equivalent over \GBC\ to the principle of elementary transfinite recursion $\ETRord$ for transfinite class recursions of length $\Ord$; to the existence of various kinds of truth predicates and iterated truth-predicates; to the existence of Boolean completions for any separative class partial order; to a class-join separation principle; and to the principle of determinacy for clopen class games of rank at most $\Ord+1$. In addition, by separating the class forcing theorem from the nearby theories of figure~\ref{Figure.Theories}, placing it strictly between the theory with $\ETR_\alpha$ simultaneously for all ordinals $\alpha$ and the theory $\ETR_{\Ord\cdot\omega}$, we locate it finely in the hierarchy of second-order set theories.

\goodbreak
\begin{maintheorem*}
 The following are equivalent over \Godel-Bernays set theory \GBC.
 \begin{enumerate}[\rm(1)]
   \item\label{Main.atomic} The atomic class forcing theorem: every class forcing notion admits forcing relations for atomic formulas $$p\forces\sigma=\tau\qquad\qquad p\forces\sigma\in\tau.$$
   \item\label{Main.scheme} The class forcing theorem scheme: for each first-order formula $\varphi$ in the forcing language, with finitely many class names $\dot \Gamma_i$, there is a forcing relation applicable to this formula and its subformulas
       $$p\forces\varphi(\vec \tau,\dot\Gamma_0,\ldots,\dot\Gamma_m).$$ 
   \item\label{Main.Uniform-first-order} The uniform first-order class forcing theorem: every class forcing notion $\P$ admits a uniform forcing relation $$p\forces\varphi(\vec \tau,\dot\Gamma_0,\ldots,\dot\Gamma_m)$$ applicable to all assertions \hbox{$\varphi\in\mathcal{L}_{\omega,\omega}(\in,V^\P,\dot\Gamma_0,\ldots,\dot\Gamma_m)$} in the first-order forcing language with finitely many class names.
   \item\label{Main.Uniform-infinitary} The uniform infinitary class forcing theorem: every class forcing notion $\P$ admits a uniform forcing relation $$p\forces\varphi(\vec \tau,\dot\Gamma_0,\ldots,\dot\Gamma_m)$$ applicable to all assertions $\varphi\in\mathcal{L}_{\Ord,\Ord}(\in,V^\P,\dot\Gamma_0,\ldots,\dot\Gamma_m)$ in the infinitary forcing language with finitely many class names.
   \item\label{Main.Names-for-truth-predicate} Names for truth predicates: every class forcing notion $\P$ has a class name $\dot\T$ and a forcing relation for which $\one\forces\dot\T$ is a truth-predicate for the first-order forcing language with finitely many class names $\mathcal{L}_{\omega,\omega}(\in,V^\P,\dot\Gamma_0,\ldots,\dot\Gamma_m)$.
   \item\label{Main.Boolean-completion} Boolean completions: Every class forcing notion $\P$, that is, every separative class partial order, admits a Boolean completion $\B$, a set-complete class Boolean algebra into which $\P$ densely embeds.
   \item\label{Main.Class-join-separation} The class-join separation principle plus $\ETRord$-foundation.
   \item\label{Main.Truth-Ord-omega} For every class $A$, there is a truth predicate for $\mathcal{L}_{\Ord,\omega}(\in,A)$.
   \item\label{Main.Truth-Ord-Ord} For every class $A$, there is a truth predicate for $\mathcal{L}_{\Ord,\Ord}(\in,A)$.
   \item\label{Main.Iterated-truth} For every class $A$, there is an $\Ord$-iterated truth predicate for $\mathcal{L}_{\omega,\omega}(\in,A)$.
   \item\label{Main.Clopen-determinacy} The principle of determinacy for clopen class games of rank at most $\Ord+1$.
   \item\label{Main.ETRord} The principle $\ETRord$ of elementary transfinite recursion for $\Ord$-length recursions of first-order properties, using any class parameter.
 \end{enumerate}
\begin{figure}[h]\rm % this one adjusted to have 12 implies 7
  \begin{tikzpicture}[scale=.8,xscale=2,vertex/.style={draw,circle,black,scale=.6,fill=white,line width=.5pt}]
     % draw the nodes
     \draw  (180:2) node[vertex] (1) {1}
            (135:2) node[vertex] (2) {2}
            (90:2) node[vertex] (3) {3}
            (-.67,1.33) node[vertex] (5) {5}
            (45:2) node[vertex] (4) {4}
            (-1.33,0) node[vertex] (6) {6}
            (-.67,0) node[vertex] (7) {7}
            (-90:2) node[vertex] (8) {8}
            (0,0) node[vertex] (9) {9}
            (-45:2) node[vertex,scale=.85] (10) {10}
            (.67,0) node[vertex,scale=.85] (11) {11}
            (0:2) node[vertex,scale=.85] (12) {12}
            ;
     % draw the arrows
     \draw[-{>[scale=.5]},>=Stealth,thick]
                        %(1) edge[bend right=26]
                        %       node {\hyperref[Theorem.Atomic-implies-scheme]{\phantom{xxx}}} (2)
                        (1) edge[{<[scale=.5]}-{>[scale=.5]},>=Stealth,bend right=60,looseness=1.5] node
                            {\hyperref[Section.Boolean-completions]{\phantom{xx}}} (6)
                        % (1) edge[bend left=0,thin,OliveGreen!50!black] (2)
                        (12.100) edge[bend right=16] node {\hyperlink{ETRord-Uniform-infinitary}{\phantom{xxx}}} (4)
                        (1.-80) edge[red!70!black,bend right=41]
                              node {\hyperref[Section.Forcing-theorem-implies-truth-for-L_Ord,omega]{\phantom{xxx}}} (8.west)
                        (4) edge[blue!70!black,bend right=16]
                              node {\hyperlink{ETRord-Uniform-infinitary}{\phantom{xxx}}} (3.east)
                        (3) edge[blue!70!black,bend right=16]
                              node {\hyperlink{ETRord-Uniform-infinitary}{\phantom{xxx}}} (2)
                        (2) edge[blue!70!black,bend right=16]
                              node {\hyperlink{ETRord-Uniform-infinitary}{\phantom{xxx}}} (1)
                        (3) edge[{<[scale=.5]}-{>[scale=.5]},>=Stealth,bend left=40,looseness=1.2]
                              node {\hyperref[Section.Forcing-relation-as-name-for-truth-predicate]{\phantom{xxx}}} (5)
                        (9) edge[blue!70!black,bend right=30] node
                             {\hyperref[Section.ETRord-implies-truth-for-L_Ord,Ord]{\phantom{xxx}}} (8.120)
                        (8) edge[red!70!black,bend right=16] node
                              {\hyperref[Theorem.Truth-Ord-omega-implies-Iterated-truth]{\phantom{xxx}}} (10)
                        (7) edge[bend right=25,out=-35]
                              node {\hyperref[Section.Class-join-separation]{\phantom{xxx}}} (8.150)
                        (12.110) edge[bend right=40,dash pattern=on 6pt off
                                      1pt,thin,looseness=.95,OliveGreen!50!black,out=-70] node {\hyperref[Theorem.ETRord-implies-forcing-relation-atomic]{\phantom{xxx}}} (1)
                        (12.125) edge[bend right=40,out=-55]
                              node {\hyperref[Section.Class-join-separation]{\phantom{xxx}}}  (7)
                        (12.150) edge[bend right=40,out=-40] node
                              {\hyperref[Section.ETRord-implies-truth-for-L_Ord,Ord]{\phantom{xxx}}} (9)
                        (12.175) edge[bend right=20]
                             node {\hyperref[Theorem.ETRord-implies-clopen-determinacy]{\phantom{xxx}}} (11)
                        (11) edge[red!70!black,bend right=25,out=-35] node {\hyperref[Theorem.Clopen-determinacy-implies-truth]{\phantom{xxx}}} (8.north)
                        (10) to[red!70!black,bend right=16] node {\hyperref[Section.Ord-iterated-truth-implies-ETRord]{\phantom{xxx}}} (12)
                        ;
  \end{tikzpicture}
\end{figure}
\end{maintheorem*}%
We shall prove the theorem by establishing the complete cycle of implications pictured above. For convenience, the implication diagram is clickable, with each implication arrow linking to the corresponding theorem where the implication is proved. In addition, all subsequent mentions of an implication, such as \hbox{$\ref{Main.Truth-Ord-omega}\to\ref{Main.Iterated-truth}$}, link to the corresponding theorem statements. Precise details for the terms appearing in the theorem statements appear in the various sections below where the corresponding implications are proved. The red implication arrows indicate more substantial or critical implication arguments; blue arrows correspond to the essentially immediate implications; and the dashed green arrow indicates implication $\ref{Main.ETRord}\to\ref{Main.atomic}$, which although not needed to complete the cycle, is nevertheless used in the proof of theorem~\ref{Theorem.ETRord-implies-uniform-forcing-relation}, establishing the implication $\ref{Main.ETRord}\to\ref{Main.Uniform-infinitary}$.

We should like particularly to emphasize that statement~\ref{Main.scheme} is a scheme over the formulas $\varphi$ that are finite in the meta-theory, stating as a scheme that for each such formula, there is a forcing relation class that works with that formula and its subformulas. Statements~\ref{Main.Uniform-first-order} and~\ref{Main.Uniform-infinitary}, in constrast, are not schemes, but assert in each case that there is a single uniform relation that works with all formulas simultaneously. In the set-forcing \ZFC\ context, we are used to having the forcing relations $p\forces\varphi(\vec\tau)$ available only as a scheme, a separate relation for each formula $\varphi$, and because of Tarski's theorem on the non-definability of truth, it follows that one cannot prove in \ZFC\ that there is a single unified forcing relation $\forces_\P$ that works with all formulas (although by theorem~\ref{Theorem.Atomic-to-quantifier-free-forcing-relation} one can do this in the quantifier-free infinitary case). In the case of class forcing, however, our main theorem shows that if every class forcing notion $\P$ admits forcing relations merely for atomic formulas, then in fact they all have fully uniform forcing relations $\forces_\P$, and not only for the first-order forcing languages, but for the infinitary forcing languages, as in statement~\ref{Main.Uniform-infinitary}.

Let us remark specifically on the role of the global axiom of choice in this analysis. The base theory for the subject is \Godel-Bernays set theory \GBC, which includes the global axiom of choice, the assertion that there is a class well-ordering of the universe. But actually, none of the arguments in the proof of the main theorem require the global axiom of choice, as opposed to the ordinary axiom of choice for sets, with the exception of the implications to the clopen determinacy assertion of statement~\ref{Main.Clopen-determinacy}. In particular, our arguments show that all the statements in the main theorem except statement~\ref{Main.Clopen-determinacy} are equivalent over the theory \GB+\AC, which has the axiom of choice only for sets. Meanwhile, the clopen determinacy assertion of statement~\ref{Main.Clopen-determinacy} in the main theorem implies the global axiom of choice, by the folklore result mentioned in~\cite[theorem~4]{GitmanHamkins2016:OpenDeterminacyForClassGames}, and so global choice is required if one wants to include statement~\ref{Main.Clopen-determinacy}.

For definiteness, in this article we say $\P$ is a \emph{class forcing notion} if $\P$ is a separative class pre-order.
Note that the axiom of global choice allows us to form the separative quotient of any class pre-order.

\section{$\ETRord$ implies the class forcing theorem scheme}\label{Section.ETR_Ord-implies-forcing-theorem}

In this section, we shall prove the implications $\ref{Main.ETRord}\to\ref{Main.atomic}$ and \hbox{$\ref{Main.atomic}\to\ref{Main.scheme}$} in the main theorem. Let's begin by defining the notions carefully.

\begin{definition}\label{Definition.ETRord}\rm
The principle of \emph{elementary transfinite recursion} $\ETRord$ for recursions of length $\Ord$, is the scheme asserting of any first order formula $\varphi(x,X,A)$ with a class parameter $A$, that there is a class $S\of \Ord\times V$ that is a solution of the following recursion
    $$S_\alpha=\set{x\mid \varphi(x,S\restrict \alpha,A)},$$
where $S_\alpha=\set{x\mid \<\alpha,x>\in S}$ denotes the $\alpha^{\rm th}$ slice of $S$ and  $S\restrict\alpha=S\intersect(\alpha\times V)$ is the part of the solution prior to stage $\alpha$.
\end{definition}

Thus, the axiom asserts that we may undertake transfinite recursive definitions of classes by recursions of length $\Ord$. In general, in \GBC\ one may not necessarily undertake class recursions even merely of length $\omega$, since first-order truth, of course, is defined by the Tarskian recursion on formulas, and this recursion has length merely $\omega$; but in \GBC, if consistent, one cannot prove the existence of a truth-predicate for first-order truth and therefore this recursion may have no solution. So even $\ETR_\omega$, which asserts that class recursions of length $\omega$ have solutions, is not provable in \GBC, if \GBC\ is consistent, and consequently neither is $\ETRord$.

Meanwhile, $\ETRord$ is a consequence of the full principle $\ETR$ of elementary transfinite recursion, which allows recursion along any class well-founded relation, including relations much taller than $\Ord$, and this principle is strictly weaker than $\GBC+\Pi^1_1$-comprehension~\cite{Sato2014:Relative-predicativity-and-dependent-recursion-in-second-order-set-theory-and-higher-order-theories}, which is itself strictly weaker than Kelley-Morse set theory \KM. Gitman and Hamkins~\cite{GitmanHamkins2016:OpenDeterminacyForClassGames} proved that \ETR\ is equivalent over \GBC\ to the principle of determinacy for clopen class games.

The idea of~\cite[lemma~7]{GitmanHamkins2016:OpenDeterminacyForClassGames} shows that the principle $\ETRord$ is equivalently formulated in terms of recursions along arbitrary well-founded class relations of rank $\Ord$.

\begin{definition}\label{Definition.Forcing-relation-atomic}\rm A class forcing notion $\P$ \emph{admits forcing relations for atomic formulas}, if there are relations $$p\forces\sigma\in\tau,\qquad\qquad p\forces\sigma\of\tau,\qquad\qquad p\forces \sigma=\tau$$
which obey the following recursive properties:
\begin{enumerate}[(a)]
  \item $p\forces \sigma\in\tau$ if and only if there is a dense class of conditions $q\leq p$ for which there is $\<\rho,r>\in\tau$ with $q\leq r$ and $q\forces \sigma=\rho$.
  \item $p\forces \sigma=\tau$ if and only if $p\forces\sigma\of\tau$ and $p\forces \tau\of\sigma$.
  \item $p\forces \sigma\of\tau$ if and only if whenever $\<\rho,r>\in\sigma$ and $q\leq p,r$ then $q\forces\rho\in\tau$.

%  \item $p\forces \sigma\of\tau$ if and only if whenever $\<\rho,r>\in\sigma$ and $q'\leq p,r$, there is $q\leq q'$ with $q\forces\rho\in\tau$.
\end{enumerate}
\end{definition}

Since the formulas are distinguished syntactically, we may unify the three forcing relations into a single relation $\forces$, applied to any atomic assertion. One may take statement (c) as a definition of the relation $p\forces\sigma\of\tau$ in terms of the forcing relation for $\in$, and in this case statements (a) and (b) are expressible as a recursion solely in terms of $p\forces\sigma\in\tau$ and $p\forces\sigma=\tau$. So the use of $\of$ here is merely a convenience.

\begin{theorem}\label{Theorem.ETRord-implies-forcing-relation-atomic}
  The principle $\ETRord$ of elementary transfinite recursion for class recursions of length $\Ord$ implies that every class forcing notion $\P$ admits a forcing relation for atomic formulas.
\end{theorem}

This will establish implication $\ref{Main.ETRord}\to\ref{Main.atomic}$ in the main theorem.

\begin{proof}
The main point is that having a forcing relation for atomic formulas, by definition, is to have a solution of the recursion that is expressed by definition~\ref{Definition.Forcing-relation-atomic}. Since this is an $\in$-recursion on the $\P$-names, we may organize it as a recursion of length $\Ord$, using the natural $\Ord$-ranking of pairs of names $\<\sigma,\tau>$, which are ordered first by the maximum of their ranks and then lexically by rank. We may now place the forcing relation $p\forces\sigma\in\tau$ and $p\forces \sigma=\tau$ on the $\alpha^{\rm th}$ slice, when the pair $\<\sigma,\tau>$ has rank $\alpha$ with respect to that relation; the forcing relation for such names is defined in terms of the forcing relation on preceding pairs of names. So $\ETRord$ is sufficient to find a solution of the recursion.
\end{proof}

Now, let us explain how to extend the forcing relation beyond the atomic formulas. A class $\P$-name $\dot\Gamma$ is simply a class of pairs $\<\rho,p>$, where $\rho$ is a $\P$-name and $p\in\P$; so, it is a class that is a $\P$-name. The canonical name for the generic filter is the class $\dot G=\set{\<\check p,p>\mid p\in\P}$. We denote by $\mathcal{L}_{\omega,\omega}(\in,V^\P,\dot\Gamma_0,\ldots,\dot\Gamma_m)$ the usual first-order forcing language, allowing the names $\sigma\in V^\P$ as constants and allowing finitely many class name parameters $\dot\Gamma_i$. This notation will mesh with the more general notation we introduce in definition~\ref{Definition.Infinitary languages} for the various infinitary languages.

\begin{definition}\label{Definition.Forcing-relation-scheme}\rm
A class forcing notion $\P$ admits a forcing relation for a collection of first-order formulas, closed under subformulas, if there is a relation $\forces$ obeying the following recursive properties, for the formulas on which it is defined:
\begin{enumerate}[(a)]
  \item The forcing relation $\forces$ is defined on atomic formulas $\sigma=\tau$ and $\sigma\in\tau$ in accordance with definition~\ref{Definition.Forcing-relation-atomic};
  \item For a class name $\dot\Gamma$, we have $p\forces\sigma\in\dot\Gamma$ if and only if there is a dense class of $q\leq p$ for which there is $\<\tau,r>\in\dot\Gamma$ with $q\leq r$ and $q\forces\sigma=\tau$.
  \item $p\forces \varphi\wedge\psi$ if and only if $p\forces\varphi$ and $p\forces\psi$;
  \item $p\forces \neg\varphi$ if and only if there is no $q\leq p$ with $q\forces \varphi$; and
  \item $p\forces \forall x\, \varphi(x)$ if and only if $p\forces\varphi(\tau)$ for every $\P$-name $\tau$.
   % exists case: there is a dense class of $q\leq p$ for which there is a $\P$-name $\tau$ with $q\forces \varphi(\tau)$.
\end{enumerate}
\end{definition}

What we mean is that in each case, if the left-hand side of the equivalence is defined, then the relations appearing on the right hand side are defined and furthermore, are defined in such a way that fulfills the equivalence. We say that $\P$ admits a forcing relation for a formula $\varphi$, if it admits a forcing relation defined on a collection of formulas including all instances of $\varphi(\vec\tau)$ for any choice of $\P$-names $\vec\tau$.

We should like particularly to emphasize that this definition, as well as the definition of what it means to have a forcing relation for atomic formulas, makes no reference whatsoever to generic filters or to genericity of any kind or to the truth of any formula in any forcing extension. Rather, for a forcing relation to exist means, by definition, precisely that there is a class relation $\forces_\P$ exhibiting the recursive properties expressed in definitions~\ref{Definition.Forcing-relation-atomic} and~\ref{Definition.Forcing-relation-scheme} (and later, for the infinitary language case, definition~\ref{Definition.Forcing-relation-infinitary}). These recursive properties are entirely first-order expressible properties of the class relation $\forces_\P$, and the question of whether $\forces_\P$ has those properties can be answered entirely in the ground model. The question for a given forcing notion $\P$ is whether or not there is indeed a class relation $\forces_\P$ exhibiting those recursive properties.

\begin{theorem}\label{Theorem.Atomic-implies-scheme}
 If a forcing notion $\P$ admits a forcing relation for atomic formulas, then it admits a forcing relation for any particular first-order formula $\varphi$ in the forcing language $\mathcal{L}_{\omega,\omega}(\in,V^\P,\dot\Gamma_0,\ldots,\dot\Gamma_m)$.
\end{theorem}

This establishes $\ref{Main.atomic}\to\ref{Main.scheme}$ in the main theorem, which although not necessary for the main cycle of implications, will be used in subsequent arguments.

\begin{proof} This is also proved in~\cite{HolyKrapfLuckeNjegomirSchlicht2016:Class-forcing}. This is a theorem scheme, proved by meta-theoretic induction on the formula $\varphi$. Given the forcing relation $\forces$ defined on atomic formulas $\sigma=\tau$ and $\sigma\in\tau$, one may proceed simply to define the forcing relation for any given first-order formula in the meta-theory. For any particular formula $\varphi$, we may form the finite set of subformulas of $\varphi$, plus the atomic formulas, and then simply apply the recursive definitions expressed by the requirements of definition~\ref{Definition.Forcing-relation-scheme}. For an actual formula, this recursion takes place in the meta-theory, and so we get the desired forcing relation in finitely many recursive steps.
\end{proof}

Because the induction in the proof of theorem~\ref{Theorem.Atomic-implies-scheme} takes place in the meta-theory, the result applies only to formulas $\varphi$ that are finite in the meta-theory. We cannot use this theorem directly, for example, to get a forcing relation for nonstandard formulas in a nonstandard model of \GBC. Nevertheless, the main theorem shows that the principle $\ETRord$ implies that indeed we can have a uniform forcing relation applying to such formulas in such a model, as in statement~\ref{Main.Uniform-first-order} of the main theorem.

\begin{lemma}\label{Lemma.Forcing-relations-respects-logical-consequence}
Suppose that $\P$ is a class forcing notion and $\forces$ is a forcing relation defined on the relevant formulas.
 \begin{enumerate}
   \item If $p\forces\varphi$ and $q\leq p$, then $q\forces\varphi$.
   \item If it is dense below $p$ to force $\varphi$, then $p\forces\varphi$.
   \item If $\varphi\to\psi$ is a logical validity and $p\forces\varphi$, then $p\forces\psi$.
 \end{enumerate}
\end{lemma}

\begin{proof}
Statements (1) and (2) are each proved for atomic formulas by induction on names, and then easily extended to all formulas by induction on formulas. For example, for the negation case of statement (2), if it is dense below $p$ to force $\neg\varphi$, then for any $q\leq p$ there is $r\leq q$ with $r\forces\neg\varphi$, which means by (1) that $q$ cannot force $\varphi$, and so $p\forces\neg\varphi$, as desired.

Statement (3) is proved by induction on proofs with respect to a standard deduction system. It is easy to verify, for example, that forcing respects \emph{modus ponens}:
 $$\text{if }p\forces\varphi\text{ and }p\forces\varphi\rightarrow \psi,\qquad\text{ then }p\forces \psi.$$
Forcing respects substitution, since by definition, $p\forces\forall x\varphi(x)$ just in case $p\forces\varphi(\tau)$ for any particular name $\tau$. Also, if $p\forces \varphi$ and $x$ is not a variable in $\varphi$, then $p\forces\forall x\varphi$.

One can prove by induction on names that every condition forces every instance of the atomic equality axioms:
\begin{align*}
    \sigma&=\sigma\\
    \sigma&=\tau\rightarrow \tau=\sigma\\
    \sigma&=\tau\rightarrow (\tau=\rho\rightarrow \sigma=\rho)\\
    \sigma&=\tau\rightarrow (\rho\in\sigma\leftrightarrow \rho\in\tau).
\end{align*}
We leave to the reader the further elementary exercises to prove for every condition $p$ that
\begin{align*}
  p&\forces\forall x(\varphi\rightarrow\psi)\rightarrow (\forall x\varphi\rightarrow\forall x\psi) \\
  p&\forces\varphi\rightarrow (\psi\rightarrow \varphi) \\
%  p&\forces\neg\neg\varphi\quad\text{ implies }\quad p\forces\varphi \\
  p&\forces (\neg\psi\rightarrow\neg\varphi)\rightarrow (\varphi\rightarrow \psi) \\
  p&\forces [\varphi\rightarrow (\psi\rightarrow \theta)]\rightarrow[(\varphi\rightarrow\psi)\rightarrow(\varphi\rightarrow\theta)]
\end{align*}
In each case, one can prove the statement by considering the definition of what it means to be a forcing relation and by using statements (1) and (2) in a density argument. Since we have therefore observed that forcing respects the axioms and rules of a complete deduction system, statement (3) now follows by induction on proofs.
\end{proof}

The lemma also holds by essentially similar arguments for the infinitary languages we introduce in section~\ref{Section.Infinitary-languages}. For example, the forcing relation respects the infinitary conjunction rule and the infinitary quantifier rules.

\section{Constructing actual forcing extensions}

Before proceeding to the rest of the main theorem, we should like to clarify some issues concerning the forcing relation and class forcing and the construction of forcing extensions. It has been traditional to highlight what has been called the forcing theorem, which explains how the forcing relation interacts with truth in a forcing extension. What we would like to do here is to explain how that part of the forcing theorem is a consequence of the existence of the forcing relation, which we defined here as a solution of the forcing relation recursion. The central question with regard to a class forcing notion $\P$, therefore, becomes whether indeed one has such a forcing relation; for if one does, then it will interact with truth in the forcing extension in the desired manner.

In order to show this, let us first review how one constructs a forcing extension for a forcing notion $\P$. We shall explain how to construct forcing extensions of an arbitrary model $M$ of \GBC, without assuming that that model is countable or transitive.

Given a class forcing notion $\P$ inside a model $M$ of \GBC, let's say that a filter $G\of\P$ is $M$-generic, if $G$ meets every dense subclass $D\of\P$ that is in $M$. Suppose that we have a forcing relation $\forces$ available in $M$ for $\P$. For any such $G$, we may define the following relations on the $\P$-names available in $M$:
\begin{equation*}
  \begin{split}
     \sigma=_G\tau & \quad\Iff\quad \exists p\in G\quad p\forces\sigma=\tau \\
     \sigma\in_G\tau  & \quad\Iff\quad\exists p\in G\quad p\forces \sigma\in\tau.
  \end{split}
\end{equation*}
Equivalently, using the Boolean values described in theorem~\ref{Theorem.Atomic-iff-Boolean-completion}, we have defined
\begin{equation*}
  \begin{split}
     \sigma=_G\tau & \quad\Iff\quad \boolval{\sigma=\tau}\in G\\
     \sigma\in_G\tau  & \quad\Iff\quad \boolval{\sigma\in\tau}\in G.
  \end{split}
\end{equation*}
It follows easily using lemma~\ref{Lemma.Forcing-relations-respects-logical-consequence} that $=_G$ is an equivalence relation and a congruence with respect to $\in_G$. Let $M[G]$ denote the collection of equivalence classes $[\sigma]_{=_G}$, equipped with the relation induced by $\in_G$. This amounts to the same quotient procedure one undertakes with the Boolean-valued model approach to forcing and the Boolean ultrapower (discussed in~\cite{HamkinsSeabold:BooleanUltrapowers}).

It is well-known that $\<M[G],\in_G>$ is not necessarily a model of \ZFC, since with class forcing one can, for example, collapse all cardinals to $\omega$, which of course contradicts \ZFC; indeed, one can even collapse $M$ itself to become countable, as will happen with the forcing $\FA$ that we consider in section~\ref{Section.Forcing-theorem-implies-truth-for-L_Ord,omega}. Nevertheless, $\<M[G],\in_G>$ is a structure of some kind and has its theory, whatever it may be, and so it still makes sense to inquire about which statements in that theory are true in $M[G]$ or forced by a condition, and so on.

If we had equipped the forcing relation with a class name $\dot\Gamma$, then we define its extension in $M[G]$ by
 $$\Gamma([\sigma])\quad\Iff\quad \exists p\in G\ p\forces\sigma\in\dot\Gamma,$$
which is the same as saying that there is a dense class of conditions $q\leq p$ for which there is $\<\rho,r>\in\dot\Gamma$ with $q\leq r$ and $q\forces\sigma=\rho$. For example, if we use the canonical name of the generic filter $\dot G$, then we have $G([\sigma])$ just in case there is $p\in G$ with $p\forces\sigma\in\dot G$. If $G$ is $M$-generic, this is equivalent to saying that there is some $p,r\in G$ with $p\leq r$ and $p\forces\sigma=\check r$.

\begin{theorem}\label{Theorem.Truth-theorem}
 Suppose that $M$ is a model of \GBC\ with a class forcing notion $\P$ that admits a forcing relation in $M$ for a first-order formula $\varphi$, which is standard in the meta-theory, in a forcing language with finitely many class names $\dot\Gamma_i$ and $G\of\P$ is $M$-generic. Then
  $$\<M[G],\in_G,\Gamma_0,\ldots,\Gamma_n>\satisfies\varphi([\tau])\quad\text{ if and only if }\quad\exists p\in G\ p\forces\varphi(\tau).$$
\end{theorem}

\begin{proof}
This is proved by induction in the meta-theory on the formula $\varphi$, which is why we need $\varphi$ to be standard in the meta-theory. The theorem is true for atomic formulas basically by definition of the relations $=_G$ and $\in_G$ and the definitions of the extensions $\Gamma_i$ in $M[G]$ for any of the class names $\dot\Gamma_i$. If the theorem statement is true for $\varphi$, then it is also true for $\neg\varphi$, since the class of conditions $p$ with $p\forces\varphi$ or $p\forces\neg\varphi$ is easily seen to be dense, and so there is such a condition in $G$. If the theorem statement is true for $\varphi$ and $\psi$, then it is also true for the conjunction $\varphi\wedge\psi$, because the class of conditions $p$ that either force both $\varphi$ and $\psi$ or else force one of the negations $\neg\varphi$ or $\neg\psi$ is dense, and so there is such a $p\in G$, which then gives the theorem for the conjunction. If the theorem is true for $\varphi$, then it is true for $\forall x\, \varphi(x)$, since the class of conditions $p$ that either force $\neg\varphi(\tau)$ for some $\tau$ or force $\varphi(\tau)$ for all such names $\tau$, is dense, and so there is such a condition $p\in G$, which then gives the theorem for $\forall x\, \varphi(x)$, as desired.
\end{proof}

The argument can just as easily handle assertions $\varphi$ in the infinitary languages $\mathcal{L}_{\Ord,\Ord}(\in,A)$, introduced in definitions~\ref{Definition.Infinitary languages} and~\ref{Definition.Forcing-relation-infinitary}, provided that the formulas are actually well-founded in the meta-theory.

We should like to emphasize that although theorem~\ref{Theorem.Truth-theorem} works with nonstandard models $M$ of \GBC, we do still have a standardness assumption for the formula $\varphi$, since one proves the theorem by induction on formulas in the meta-theory. We shall later explain how one can move beyond that to a setting accommodating nonstandard formulas $\varphi$ by using theorem~\ref{Theorem.Forcing-relation-as-name-for-truth-predicate} to construct suitable truth predicates in the extension.

\section{Forcing relation as name for a truth predicate}\label{Section.Forcing-relation-as-name-for-truth-predicate}

We shall now prove the equivalence $\ref{Main.Uniform-first-order}\iff\ref{Main.Names-for-truth-predicate}$ in the main theorem. Let us begin by making precise what it means to have a truth predicate for first-order truth.

\begin{definition}\label{Definition.Truth-predicate-first-order}\rm
A \emph{truth predicate} for first-order truth (also known as a \emph{satisfaction class}), with a class parameter $A$, is a class $\T$ consisting of pairs $\<\varphi,\vec a>$, where $\varphi$ is a formula in the first-order language of set theory augmented with a predicate symbol $\hat A$ for the class $A$ and $\vec a$ is a valuation mapping the free variables of $\varphi$ to corresponding set parameters, such that the following recursion is satisfied:
\begin{enumerate}[(a)]
  \item $\T$ judges the truth of $\{=,\in,\hat A\}$-atomic statements correctly:
            \begin{align*}
              \T(x=y,\<a,b>) &\quad \text{ if and only if }\quad a=b \\
              \T(x\in y,\<a,b>) &\quad \text{ if and only if } \quad a\in b  \\
              \T(x\in\hat A,a)\quad &\quad \text{ if and only if }\quad a\in A
            \end{align*}
  \item $\T$ performs Boolean logic correctly:
            \begin{align*}
              \T(\varphi\wedge\psi,\vec a)&\quad \text{ if and only if }\quad \T(\varphi,\vec a)\text{ and }\T(\psi,\vec a)\\
              \T(\neg\varphi,\vec a)\ &\quad \text{ if and only if }\quad \neg\T(\varphi,\vec a)
            \end{align*}
  \item $\T$ performs quantifier logic correctly:
            $$\T(\forall x\, \varphi,\vec a)\quad\text{ if and only if }\quad\forall b\, \T(\varphi,b\concat\vec a)$$
\end{enumerate}
\end{definition}

When a truth predicate exists in a model of \GBC, then it is unique, since there cannot be a least formula where the disagreement occurs. Nevertheless, classical results of Krajewski~\cite{Krajewski1974:MutuallyInconsistentSatisfactionClasses} show that there are (necessarily nonstandard) models of \ZFC\ that admit different incompatible truth predicates as $\ZFC$-amenable classes (see discussion and further results in~\cite{HamkinsYang:SatisfactionIsNotAbsolute}). Similar reasoning produces models of \ZFC\ with different incompatible forcing relations for a given forcing notion, even for set forcing, each of them $\ZFC$-amenable but not jointly $\ZFC$-amenable.

Let us introduce the notation $\mathop{\rm op}(\sigma,\tau)$ for the canonical name for the ordered pair of $\sigma$ and $\tau$; technically, if we use the Kuratowski definition of the ordered pair, then $\mathop{\rm op}(\sigma,\tau)=\{\<\{\<\sigma,\one>\},\one>, \<\{\<\sigma,\one>,\<\tau,\one>\},\one>\}$.

\begin{theorem}\label{Theorem.Forcing-relation-as-name-for-truth-predicate}
 For any class forcing notion $\P$, and any finitely many class name parameters $\dot\Gamma_0,\ldots,\dot\Gamma_m$, the following are equivalent:
 \begin{enumerate}[\rm(i)]
   \item $\P$ admits a uniform forcing relation for the first-order forcing language  $\mathcal{L}_{\omega,\omega}(\in,\dot\Gamma_0,\ldots,\dot\Gamma_m)$.
   \item There is a class name $\dot\T$ and a forcing relation defined on the following statement, for which
      $$\one\forces\dot\T\text{ is a truth predicate for }\mathcal{L}_{\omega,\omega}(\in,\dot\Gamma_0,\ldots,\dot\Gamma_m).$$
 \end{enumerate}
\end{theorem}

The main lesson of this argument should be that a uniform forcing relation essentially \emph{is} the name of a truth predicate.

\begin{proof}
$\rm(i)\to(ii)$ Suppose that we have a uniform forcing relation $\forces$ for $\P$ in the first-order forcing language $\mathcal{L}_{\omega,\omega}(\in,\dot\Gamma_0,\ldots,\dot\Gamma_m)$. Let
  $$\dot\T\quad=\quad\set{\<\mathop{\rm op}(\check\varphi,\vec\tau),p>\mid p\forces\varphi(\vec\tau)},$$
which is a class name for the pairs $\<\varphi,\vec a>$ such that $\varphi(\vec a)$ will be true in the extension. Since we have the atomic forcing relation, it follows by theorem~\ref{Theorem.Atomic-implies-scheme} that we may extend $\forces$ to cover any particular first-order statement using any other fixed class parameter. In particular, we can talk about whether specific assertions about $\dot\T$ are forced, even if $\dot\T$ is not one of the $\dot\Gamma_i$ (that is, theorem~\ref{Theorem.Atomic-implies-scheme} allows us to apply the forcing relation to specific individual assertions about $\dot\T$). Using this, we claim that
  $$\one\forces\dot\T\text{ is a first-order truth predicate}.$$
This is a first-order-expressible property of the predicate $\dot\T$.

The main point is that the recursive requirements on the forcing relation are exactly what one needs to prove that $\dot\T$ is (forced to be) a truth predicate. The atomic case is clear, because a condition $p$ forces (the canonical name of) $\<x\in y,\<\sigma,\tau>>$ to be in $\dot\T$ if and only if $p\forces\sigma\in\tau$, and similarly with $=$ and with any class name parameter $\dot\Gamma$. Since $p\forces\neg\varphi$ just in case no stronger condition $q\leq p$ forces $\varphi$, it follows that the class of conditions $q$ either forcing $\varphi$ or forcing $\neg\varphi$, but never both, is dense, and so $\one$ forces that $\dot\T$ exhibits the negation requirement for a truth predicate. Similarly, since any condition $p$ forces $\bigwedge_i\varphi_i$ just in case it forces each $\varphi_i$ separately, it follows that $\one$ forces that $\dot\T$ obeys the conjunction rule. And since $p\forces\forall\vec x\, \varphi(\vec x)$ just in case $p\forces\varphi(\vec\tau)$ for all particular $\vec\tau$, it follows that $\one$ forces that $\dot\T$ obeys the quantifier requirement.

\newcommand\varforces{\Vvdash}
$\rm(ii\to i)$ Conversely, suppose that $\P$ admits a class $\P$-name $\dot\T$ and a forcing relation $\forces$ for which
  $$\one\forces\dot\T\text{ is an }\mathcal{L}_{\omega,\omega}(\in,\dot\Gamma_0,\ldots,\dot\Gamma_m)\text{-truth predicate.}$$
We take it as part of this assertion, following definition~\ref{Definition.Forcing-relation-scheme}, that the forcing relation $\forces$ is defined on all atomic formulas. For assertions $\varphi$ in the forcing language, let us define a new relation $\varforces$ as follows:
 $$p\varforces\varphi(\vec\tau)\quad\longleftrightarrow\quad p\forces \op(\check\varphi,\vec\tau)\in\dot\T.$$
That is, we say that an assertion $\varphi(\tau)$ is forced via $\varforces$ by a condition $p$, if $p$ forces that it is true according to the truth predicate $\dot\T$. On the right hand side of this definition, we use the extension of the atomic forcing relation $\forces$ as in theorem~\ref{Theorem.Atomic-implies-scheme} to any particular first-order expression in the forcing language.

First, let's notice that the two forcing relations agree on atomic formulas. Specifically, because of what it means to be a truth predicate, this relation gets the right answer for atomic formulas, since we have that $p\varforces\sigma=\tau$ if and only if $p\forces\op\bigl((x=y)^\echeck,\op(\sigma,\tau)\bigr)\in\dot\T$, which holds if and only if $p\forces\sigma=\tau$; and similarly, $p\varforces\sigma\in\tau$ if and only if $p\forces\op\bigl((x\in y)^\echeck,\op(\sigma,\tau)\bigr)\in\dot\T$, which holds if and only if $p\forces\sigma\in\tau$. So $\varforces$ and $\forces$ agree on atomic assertions (and therefore we didn't really need a new forcing symbol).

We have to verify that this relation that we have just defined satisfies the required recursive properties to be a forcing relation. We have already observed that it agrees with the forcing relation for atomic assertions. Since $\one$ forces that $\dot\T$ is a truth predicate, it follows that for any particular assertion $\varphi(\vec\tau)$, there are densely many conditions $q$ either forcing that $\varphi(\vec\tau)$ is true according to $\dot\T$ or that $\neg\varphi(\vec\tau)$ is true according to $\dot\T$, but never both. From this, it follows that $p\varforces\neg\varphi(\vec\tau)$ just in case no stronger condition forces $\varphi(\vec\tau)$. Similarly, the truth-predicate requirements on conjunctions lead to the fact that $p\varforces\varphi\wedge\psi$ if and only if $p\varforces\varphi$ and $p\varforces\psi$. And $p\varforces\forall x\,\varphi(x)$ just in case $p\varforces\varphi(\tau)$ for all $\P$-names $\tau$. So we've got a uniform forcing relation, as desired.
\end{proof}

Theorem~\ref{Theorem.Forcing-relation-as-name-for-truth-predicate} establishes the equivalence $\ref{Main.Uniform-first-order}\iff\ref{Main.Names-for-truth-predicate}$ in the main theorem on a case-by-case basis for any class forcing notion $\P$. A similar argument shows that to have a forcing relation $\forces$ defined on a fragment of the forcing language, such as the collection of subformulas of a given formula or the formulas of a certain complexity, then $\one$ forces that the corresponding predicate $\dot\T$ is a truth predicate on the same language fragment. In the case of the infinitary languages introduced in section~\ref{Section.Infinitary-languages}, an analogue of theorem~\ref{Theorem.Forcing-relation-as-name-for-truth-predicate} provides a name for a truth predicate for the language consisting of all ground-model assertions in that infinitary language.

\section{The infinitary languages}\label{Section.Infinitary-languages}

Let us now explain how to extend the forcing relation concept to the case of various infinitary languages. These languages are unified as instances of the following general definition.

\begin{definition}\label{Definition.Infinitary languages}\rm Assume $\kappa,\lambda\leq\Ord$ are infinite cardinals or $\Ord$ itself.
  \begin{enumerate}[(a)]
    \item The language $\mathcal{L}_{\kappa,\lambda}(\in,\hat A)$ has a signature consisting of the binary relation $\in$, a unary predicate symbol $\hat A$ and sufficiently many (e.g. $\lambda^{<\kappa}$ many) variable symbols $x_i$. The formulas of the language are obtained from the atomic formulas $x=y$, $x\in y$ and $x\in\hat A$ by closing under negation $\neg\varphi$, under conjunctions $\bigwedge_{j\in J}\varphi_j$ of size $|J|<\kappa$, provided that there are fewer than $\max(\omega,\lambda)$ many free variables collectively in the $\varphi_j$, and under quantification $\forall\vec x\, \varphi(\vec x)$ by quantifier blocks $\vec x=\<x_i\mid i\in I>$ of size $|I|<\lambda$.
    \item The forcing language $\mathcal{L}_{\kappa,\lambda}(\in,V^\P,\dot\Gamma_0,\ldots,\dot\Gamma_m)$, for any class forcing notion $\P$, augments the previous language with all the various $\P$-names as constant symbols and predicate symbols $\dot\Gamma_i$ for finitely many class names, closing under conjunctions of size less than $\kappa$ and quantifier blocks of size less than $\lambda$.
 \end{enumerate}
\end{definition}

This general definition has natural special cases, with which we shall be concerned:
\begin{itemize}
  \item The usual first-order language of set theory $\mathcal{L}_{\omega,\omega}(\in)$;
  \item the usual first-order forcing language $\mathcal{L}_{\omega,\omega}(\in,V^\P,\dot G)$;
  \item the quantifier-free infinitary forcing language $\mathcal{L}_{\Ord,0}(\in,V^\P,\dot G)$;
  \item the partial infinitary language $\mathcal{L}_{\Ord,\omega}(\in,\hat A)$;
  \item the full infinitary language $\mathcal{L}_{\Ord,\Ord}(\in,\hat A)$.
  \item the full infinitary forcing language $\mathcal{L}_{\Ord,\Ord}(\in,V^\P,\dot\Gamma_0,\ldots,\dot\Gamma_m)$.
\end{itemize}

Every formula in each of these languages is determined by its well-founded parse tree, which specifies at each node how the formula was constructed from its various subformulas, and one may undertake inductive proofs on formulas by means of these parse trees. Indeed, let us simply identify every formula $\varphi$ with its parse tree.

Extending definitions~\ref{Definition.Forcing-relation-atomic} and~\ref{Definition.Forcing-relation-scheme}, we now define what it means to have uniform forcing relations for these languages.

\begin{definition}\label{Definition.Forcing-relation-infinitary}\rm
If $\P$ is a class forcing notion and $\mathcal{L}(\in,V^\P,\dot\Gamma_0,\ldots,\dot\Gamma_m)$ is one of the forcing languages mentioned above, then we say that $\P$ admits a uniform forcing relation for this language, if there is a relation $\forces$ obeying the following recursive properties:
\begin{enumerate}[(a)]
  \item The forcing relation $\forces$ is defined on all atomic formulas $\sigma=\tau$, $\sigma\in \tau$, and  $\sigma\in\dot\Gamma_i$ and fulfills the requirements expressed by definitions~\ref{Definition.Forcing-relation-atomic} and~\ref{Definition.Forcing-relation-scheme} for these atomic formulas;
  \item for conjunctions in the language, $p\forces \bigwedge_{j\in J}\varphi_j$ if and only if $p\forces\varphi_j$ for each $j$;
  \item $p\forces \neg\varphi$ if and only if there is no $q\leq p$ with $q\forces \varphi$;
  \item if the language allows quantification, then $p\forces\forall x\, \varphi(x)$ if and only if $p\forces\varphi(\tau)$ for every $\P$-name $\tau$; and
  \item if the language allows infinitary blocks of quantifiers, then $p\forces \forall \vec x\, \varphi(\vec x)$, where $\vec x=\<x_i\mid i\in I>$, if and only if $p\forces\varphi(\vec\tau)$ for all sequences of $\P$-names $\vec\tau=\<\tau_i\mid i\in I>$.
\end{enumerate}
\end{definition}

If we regard disjunction $\bigvee_j\varphi_j$ as an abbreviation for $\neg\bigwedge_j\neg\varphi_j$, then the definition requires that $p\forces\bigvee_j\varphi_j$ if and only if there are densely many $q\leq p$ for each of which there is some $j$ such that $q\forces \varphi_j$, allowing different $j$ for different $q$. Similarly, if we regard $\exists x\,\varphi(x)$ as an abbreviation for $\neg\forall x\,\neg\varphi(x)$, then the definition requires that $p\forces\exists x\,\varphi(x)$ if and only if there are densely many $q\leq p$ for each of which there is some $\tau$ such that $q\forces\varphi(\tau)$.

It turns out that every class forcing relation for atomic formulas can be extended to a uniform forcing relation on the quantifier-free infinitary forcing language. Indeed, the proof of theorem~\ref{Theorem.Atomic-to-quantifier-free-forcing-relation} shows that the expressive power of the forcing language \hbox{$\mathcal{L}_{\Ord,0}(\in,V^\P,\dot G)$} does not actually exceed the expressive power of the atomic equality assertions $\dot a=\dot b$, and this is why having a forcing relation for merely the atomic assertions suffices to provide a forcing relation for $\mathcal{L}_{\Ord,0}(\in,V^\P,\dot G)$. This theorem will be used in section~\ref{Section.Boolean-completions}

\begin{theorem}\label{Theorem.Atomic-to-quantifier-free-forcing-relation}
 If a class forcing notion $\P$ admits a forcing relation for atomic formulas, then it admits a uniform forcing relation in the quantifier-free infinitary forcing language $\mathcal{L}_{\Ord,0}(\in,V^\P,\dot G)$.
\end{theorem}

\begin{proof}
This argument adapts the main ideas of~\cite[lemma~5.2, theorem~5.5]{HolyKrapfLuckeNjegomirSchlicht2016:Class-forcing}, generalizing from the context there of a countable transitive model of set theory $M$ to our development here of class forcing as an internal \GBC\ matter. The difference is that in the countable-transitive-model case, one is able to define a forcing relation externally by reference to what is true in the various extensions $M[G]$, and then use that relation in inductive arguments; but here, we must define a suitable class relation internally to \GBC\ and then prove that it fulfills the forcing-relation recursion, even in uncountable or nonstandard models and with nonstandard formulas, whose truth conditions are not necessarily sensible in the meta-theory.

To begin, we assume that $\P$ is a class forcing notion for which we have a forcing relation $\forces$ for the atomic formulas. We aim to define a forcing relation $p\forces \varphi$ for the sentences $\varphi$ in the quantifier-free infinitary forcing language $\mathcal{L}_{\Ord,0}(\in,V^\P,\dot G)$. In order to do so, we shall recursively assign to each sentence $\varphi$ in the quantifier-free infinitary forcing language $\mathcal{L}_{\Ord,0}(\in,V^\P,\dot G)$ an atomic formula of the form $\dot a_\varphi=\dot b_\varphi$ and then define the desired forcing relation $\forces$ as follows:
  $$p\forces\varphi\qquad\longleftrightarrow\qquad p\forces\dot a_\varphi=\dot b_\varphi\,.\qquad\qquad(\ast)$$
Basically, we aim to find suitable atomic equality assertions $\dot a_\varphi=\dot b_\varphi$ that track the truth and forcing conditions for any given quantifier-free infinitary assertion $\varphi$, and we shall then prove that the relation $\forces$ defined by $(\ast)$ is indeed a forcing relation.

The assignment $\varphi\mapsto\dot a_\varphi=\dot b_\varphi$ will be the end result of a certain transfinite syntactic translation process, which gradually reduces a given sentence $\varphi$ in $\mathcal{L}_{\Ord,0}(\in,V^\P,\dot G)$ to sentences appearing earlier in the (well-founded) translation hierarchy that is implicit in our construction, eventually reaching the atomic equalities $\dot a_\varphi=\dot b_\varphi$ as irreducible terminal nodes. At each step, the translation process will respect the forcing-relation requirements of definition~\ref{Definition.Forcing-relation-infinitary}.

For the first step of the translation, we systematically eliminate use of the $\dot G$ predicate, except for the check names, by applying the following transformation whenever $\sigma$ is not itself a check name:
\begin{align*}
  \sigma\in\dot G       &\qquad\mapsto\qquad
         \bigvee_{\quad p\in\P\intersect V_{\rank(\sigma)+1}}\hskip-5mm\bigl(\,\check p\in\dot G\wedge \sigma=\check p\,\bigr)
\end{align*}
Since it is not difficult to see by induction on names that $\one\forces \sigma\in\check V_{\rank(\sigma)+1}$ for any name $\sigma$, the idea of this transformation is that the only way $\sigma$ can name a condition in $\dot G$ is if it is naming one of the conditions $p\in\P\intersect V_{\rank(\sigma)+1}$. It follows that $\one$ forces the equivalence of the assertion $\sigma\in\dot G$ with its translation at the right. The result of the transformation is a formula for which the only occurrences of the $\dot G$ predicate are with check names of the form $\check p\in\dot G$.

Next, we systematically apply the infinitary de Morgan laws to push all remaining negations to the bottom of the parse tree, so that they appear, if at all, immediately in front of atomic formulas.
\begin{align*}
  \neg\bigvee_{i\in I}\varphi_i &\qquad\mapsto\qquad \bigwedge_{i\in I}\neg\varphi_i\\
  \neg\bigwedge_{i\in I}\varphi_i &\qquad\mapsto\qquad \bigvee_{i\in I}\neg\varphi_i
\end{align*}
These transformations are logical validities and therefore, if the forcing relation is defined on the formulas at the right, then it can be legitimately extended to the formulas at the left.

After this, we eliminate most of the remaining negations by systematically applying the following reductions.
\begin{align*}
  \sigma\neq\tau    &\qquad\mapsto\qquad \sigma\not\of\tau\vee \tau\not\of\sigma\\
  \sigma\not\of\tau & \qquad\mapsto\qquad \bigvee_{\<\rho,r>\in\sigma}\bigl(\check r\in\dot G\wedge \rho\notin\tau\bigr)\\
  \sigma\notin\tau &\qquad\mapsto\qquad
         \bigwedge_{\<\rho,r>\in\tau}\bigl(\check r\notin\dot G\vee\sigma\neq\rho\bigr)
\end{align*}
What we intend here is that the reductions are applied iteratively, until they can no longer be applied; an easy inductive argument on names shows that the reduction process eventually terminates. It is easy to see in each case that $\one$ forces the equivalence of each of these negated atomic formulas with its translation.

The translation process mentioned so far reduces any given sentence to a positive infinitary Boolean combination (using iterated infinitary conjunction and disjunction) of formulas of the form $\sigma=\tau$, $\sigma\in\tau$, $\check p\in\dot G$ and $\check p\notin\dot G$. In order to eliminate all but the atomic equalities, we now apply the following reductions
\begin{align*}
     \sigma\in\tau         &\qquad\mapsto\qquad \tau =\tau\union\set{\<\sigma,\one>}\\
     \check q\in\dot G     &\qquad\mapsto\qquad \set{\<\emptyset,q>} = \set{\<\emptyset,\one>} \\
     \check q\notin\dot G  &\qquad\mapsto\qquad \set{\<\emptyset,q>} = \emptyset,
\end{align*}
whose equivalences are in each case forced by $\one$.

Thus, we have transformed every sentence of the quantifier-free infinitary forcing language $\mathcal{L}_{\Ord,0}(\in,V^\P,\dot G)$ to a positive Boolean combination of atomic equalities $\sigma=\tau$. We shall now systematically apply further reductions to eliminate the need for conjunctions and disjunctions, and thereby reduce every infinitary sentence $\varphi$ to a single atomic equality $\dot a_\varphi=\dot b_\varphi$.

If $\varphi_i$ has already been mapped to the atomic equality $\dot a_i=\dot b_i$, then we eliminate the infinitary conjunction by the transformation
$$\bigwedge_{i\in I}\varphi_i\quad\mapsto\quad \set{\<\operatorname{op}(\check i,\dot a_i),\one>\mid i\in I}
   = \set{\langle\operatorname{op}(\check i,\dot b_i),\one\rangle\mid i\in I}.$$
It is easy to see that a condition $p$ forces the atomic equality $\dot a=\dot b$ at the right if and only if $p\forces\dot a_i=\dot b_i$ for all $i\in I$.

Infinitary disjunctions are a little more troublesome, but we can eliminate them by means of the following transformation.
$$\bigvee_{i\in I}\varphi_i\qquad\mapsto\qquad\dot a=\dot b,$$
where
\begin{align*}
     \dot u  \qquad&=\qquad\set{\langle\op(\check i,\dot a_i),\one\rangle\mid i\in I}\union\set{\langle\op(\check i,\dot b_i),\one\rangle\mid i\in I}\\
     \dot u_j\hskip7mm&=\qquad\set{\langle\op(\check i,\dot a_i),\one\rangle\mid i\in I}\union\set{\langle\op(\check i,\dot b_i),\one\rangle\mid i\in I,i\neq j}\\
     \dot a      \qquad&=\qquad\set{\<u_j,\one>\mid j\in I},\\
     \dot b      \qquad&=\qquad\set{\<u_j,\one>\mid j\in I}\union\set{\<\dot u,\one>}.
\end{align*}
The idea is that the names $\dot u_j$ each remove what might be a redundancy from $\dot u$, and so the names $\dot a$ and $\dot b$ will name the same set just in case there is such a redundancy. Note that if $p\forces \dot a_j=\dot b_j$ for some particular $j$, then $p\forces \dot u_j=\dot u$, because of the corresponding redundancy in $\dot u$, and consequently $p\forces \dot a=\dot b$. Conversely, if $p\forces \dot a=\dot b$, then $p\forces\dot u\in \dot a$, and so there are densely many $q\leq p$ for which there is some $j$ such that $q\forces \dot u=\dot u_j$ and consequently $q\forces\dot a_j=\dot b_j$. In short, $p\forces \dot a=\dot b$ if and only if there are densely many $q\leq p$ with $q\forces \dot a_i=\dot b_i$ for some particular $i$. Thus, our transformation respects the desired forcing relation requirement for disjunctions.

We now complete the proof of the theorem. We have described a translation of any sentence $\varphi$ in the quantifier-free infinitary forcing language $\mathcal{L}_{\Ord,0}(\in,V^\P,\dot G)$ to a corresponding atomic equality $\dot a_\varphi=\dot b_\varphi$. This translation process was an ordinary set-like recursion on formulas, which can be undertaken in \GBC\ without the need for any \ETR-like assumption. We defined the forcing relation by $(\ast)$ above, namely, $p\forces\varphi$ if and only if $p\forces \dot a_\varphi=\dot b_\varphi$. This relation $\forces$, we claim, obeys the recursive requirements of definition~\ref{Definition.Forcing-relation-infinitary}. This is proved by induction on the translation order. If $\forces$ is a forcing relation on all formulas appearing before $\varphi$ in the translation process, then because as we have observed, each step of the translation process respects the requirements of the forcing relation, it follows that $\forces$ is also a forcing relation on the sentence $\varphi$. So we have defined a uniform forcing relation on \hbox{$\mathcal{L}_{\Ord,0}(\in,V^\P,\dot G)$}, as desired.
\end{proof}

%We find it interesting to remark that one can also show, using similar kinds of reductions, that every sentence $\varphi$ in the quantifier-free infinitary forcing language $\mathcal{L}_{\Ord,0}(\in,V^\P,\dot G)$ is equivalent modulo the forcing relation to a Boolean combination of assertions of the form $\check p\in\dot G$. That is, if there is a forcing relation, then every such sentence $\varphi$ is equivalent modulo that forcing relation to a certain infinitary positive combination of assertions that various specific conditions are in the generic filter. Perhaps one can hope to turn this around, and extend any given forcing relation on atomic formulas to a forcing relation on the full quantifier-free infinitary language, defining $p\forces\varphi$ by inspecting how $p$ relates to the conditions appearing explicitly in such a representation of $\varphi$.
%

The statement in the conclusion of theorem~\ref{Theorem.Atomic-to-quantifier-free-forcing-relation}, if stated for all class forcing notions---that is, the assertion that every class forcing notion $\P$ admits a uniform forcing relation for the quantifier-free infinitary forcing language---is equivalent to all the other statements made in the main theorem, because the previous theorem shows that it is implied on a case-by-case basis by statement~\ref{Main.atomic} and conversely it clearly also implies statement~\ref{Main.atomic}. So we could actually have added this statement to the main theorem as yet another equivalent assertion.

The method of the previous theorem extends to the case of limited-complexity infinitary assertions, which allow finitely many quantifiers at the top level of the parse tree. That is, if one has a forcing relation for atomic formulas, then one can have a forcing relation for any formula having finitely many quantifiers at the front, alternating in any desired pattern, followed by a quantifier-free infinitary assertion. One first gets the forcing relation for the quantifier-free infinitary language as above, and then applies induction in the meta-theory as in theorem~\ref{Theorem.Atomic-implies-scheme}.

\hypertarget{ETRord-Uniform-infinitary}{}
Let us now prove the implication $\ref{Main.ETRord}\to\ref{Main.Uniform-infinitary}$. Meanwhile, the implications $\ref{Main.Uniform-infinitary}\to\ref{Main.Uniform-first-order}\to\ref{Main.scheme}\to\ref{Main.atomic}$ are essentially immediate.

\begin{theorem}\label{Theorem.ETRord-implies-uniform-forcing-relation}
  The principle $\ETRord$ of elementary transfinite recursion for class recursions of length $\Ord$ implies that every class forcing notion $\P$ admits a uniform forcing relation for assertions in the forcing language $\mathcal{L}_{\Ord,\Ord}(\in,V^\P,\dot\Gamma_0,\ldots,\dot\Gamma_m)$, allowing any fixed class names $\dot\Gamma_i$.
\end{theorem}

\begin{proof}
For any class forcing notion $\P$, we get the forcing relation for atomic formulas by theorem~\ref{Theorem.ETRord-implies-forcing-relation-atomic}. And now the point is that the requirements stipulated for the rest of the forcing relation by definition~\ref{Definition.Forcing-relation-infinitary} amount exactly to a recursion on formulas. That is, to have a forcing relation is to have the solution of a certain recursion, the recursion expressed by definition~\ref{Definition.Forcing-relation-infinitary}. Since every formula has a well-founded parse tree, which has some ordinal rank, we may organize this recursive definition of the uniform forcing relation as a class recursion of length $\Ord$, by recursing on the rank of the parse tree of the formula. Thus, $\ETRord$ provides a solution of this recursion, which is the desired uniform infinitary forcing relation.
\end{proof}

We should like to call attention to the contrast between theorems~\ref{Theorem.Atomic-to-quantifier-free-forcing-relation} and~\ref{Theorem.ETRord-implies-uniform-forcing-relation}. In the case of theorem~\ref{Theorem.Atomic-to-quantifier-free-forcing-relation}, we constructed a uniform forcing relation for quantifier-free infinitary assertions in the forcing language $\mathcal{L}_{\Ord,0}(\in,V^\P,\dot\Gamma_0,\ldots,\dot\Gamma_m)$ on a case-by-case basis for the forcing $\P$, without any need for $\ETRord$. This was an ordinary set-like recursion on formulas. In theorem~\ref{Theorem.ETRord-implies-uniform-forcing-relation}, however, the recursion is no longer set-like, because for the quantifier case, the relation $p\forces \forall x\, \varphi(x)$ reduces to a proper class of smaller instances $p\forces\varphi(\tau)$, even when $\varphi$ is merely first-order. Thus, this is no longer an ordinary recursion on sets, but a class recursion of length $\Ord$, which can be undertaken using $\ETRord$.

In the \ZFC\ context, set theorists have grown accustomed to having separate forcing relations for each formula $\varphi$, since it is not possible ever to have a uniform forcing relation as a definable class, as this would lead quickly to a definable truth predicate, contrary to Tarski's theorem. Nevertheless, the main theorem shows that if every class forcing notion $\P$ has its atomic forcing relations, then $\ETRord$ holds, and therefore every class forcing relation $\P$ has a fully uniform forcing relation, even in the full infinitary forcing language $\mathcal{L}_{\Ord,\Ord}(\in,V^\P,\dot\Gamma_0,\ldots,\dot\Gamma_m)$. And therefore also we get the accompanying truth predicates that these relations provide. This does not violate Tarski's theorem, because while the forcing relations exist as classes, they are not first-order definable classes.

For that reason, the `definability lemma' terminology in the literature, used to refer to the assertion that the class forcing relations exist, is somewhat misleading, because this terminology should not be interpreted as asserting that the forcing relations are actually (first-order) definable classes. Indeed, certain definable class forcing notions, such as those used in section~\ref{Section.Forcing-theorem-implies-truth-for-L_Ord,omega}, cannot have first-order definable forcing relations, even in the case of the forcing relations for atomic formulas only, although those relations can exist as \GBC\ classes. It is true, however, that when the forcing relations exist, then they are the unique class relations satisfying the forcing relation recursion, and therefore they are always first-order implicitly definable, in the sense of~\cite{HamkinsLeahy2016:AlgebraicityAndImplicitDefinabilityInSetTheory}, and therefore also second-order definable relations.

\section{Boolean completions}\label{Section.Boolean-completions}

In this section we prove the equivalence $\ref{Main.atomic}\leftrightarrow\ref{Main.Boolean-completion}$. This argument basically follows~\cite[theorem~5.5]{HolyKrapfLuckeNjegomirSchlicht2016:Class-forcing}, generalizing it from the context of countable transitive models of set theory to the general case of arbitrary models of \GBC, using our recursion conception of what it means in \GBC\ to have a forcing relation. The equivalence holds on a case-by-case basis for each class forcing notion $\P$ separately.

\begin{theorem}\label{Theorem.Atomic-iff-Boolean-completion}
 For any class forcing notion $\P$, the following are equivalent:
  \begin{enumerate}[\rm(i)]
    \item $\P$ admits a forcing relation for atomic formulas.
    \item $\P$ admits a uniform forcing relation for quantifier-free infinitary formulas in the forcing language $\mathcal{L}_{\Ord,0}(\in,\dot G)$.
    \item $\P$ admits a Boolean completion, a set-complete class Boolean algebra $\B$ into which $\P$ densely embeds.
  \end{enumerate}
\end{theorem}

\begin{proof}
The implication $\rm(i\to ii)$ is provided by theorem~\ref{Theorem.Atomic-to-quantifier-free-forcing-relation}.

For $\rm(ii\to iii)$, assume that $\P$ has a uniform forcing relation for quantifier-free infinitary formulas in the forcing language $\mathcal{L}_{\Ord,0}(\in,\dot G)$. For $\varphi,\psi$ in this language, define $\varphi\approx\psi$ just in case $\one\forces\varphi\iff\psi$. This is a class equivalence relation, and the quotient $\B=\mathcal{L}_{\Ord,0}(\in,\dot G)/\approx$, inheriting the logical structure of the language itself, is easily seen to be a set-complete Boolean algebra, as in the Lindenbaum algebra, where $\neg[\varphi]_\approx=[\neg\varphi]_\approx$ and $\bigwedge_{i\in I}[\varphi_i]_\approx=[\bigwedge_{i\in I}\varphi_i]_\approx$. (We use Scott's trick to represent each equivalence class canonically by the set of its minimal-rank members, in order to avoid the inconvenience that an equivalence class would otherwise be a proper class.) Finally, $\P$ embeds densely into $\B$ by the map $p\mapsto [\check p\in\dot G]_\approx$.

Lastly, for $\rm(iii\to i)$, suppose that a class forcing notion $\P$ has a Boolean completion $\B$, a set-complete class Boolean algebra $\B$ with a dense embedding $i:\P\to\B$. We shall prove that there is a forcing relation for atomic formulas. To do so, we define the following Boolean values, by recursion on names:
 \begin{equation*}
   \begin{split}
      \boolval{\sigma\in\tau} &\quad=\quad \bigvee_{\<\rho,r>\in\tau} \bigl(\boolval{\sigma=\rho}\wedge i(r)\bigr) \\
      \boolval{\sigma=\tau}   &\quad=\quad  \boolval{\sigma\of\tau}\wedge\boolval{\tau\of\sigma}\\
      \boolval{\sigma\of\tau} &\quad=\quad  \bigwedge_{\<\rho,r>\in\sigma}\bigl(\neg i(r)\vee\boolval{\rho\in\tau}\bigr)
   \end{split}
 \end{equation*}
This is a straightforward recursion on names, which we may undertake in \GBC, without need for any $\ETR$ assumption. We now define the corresponding forcing relation $p\forces\varphi$ if and only if $i(p)\leq\boolval{\varphi}$, for atomic $\varphi$.

It remains to check that this relation is indeed a forcing relation for atomic truth, that is, that it fulfills the recursive requirements of definition~\ref{Definition.Forcing-relation-atomic}. This is an exercise in the usual Boolean-valued reasoning. For example, if $p\forces\sigma\in\tau$, then
  $$i(p)\leq\boolval{\sigma\in\tau}=\bigvee_{\<\rho,r>\in\tau}\bigl(\boolval{\sigma=\rho}\wedge i(r)\bigr),$$
and so every $p'\leq p$ has $i(p')$ compatible with $\boolval{\sigma=\rho}\wedge i(r)$ for some $\<\rho,r>\in\tau$, which means there is some $q\leq p'$ with $q\leq r$ and $i(q)\leq \boolval{\sigma=\rho}$. Thus, there is a dense class of $q\leq p$ with $q\forces \sigma=\rho$ for some $\<\rho,r>\in\tau$ with $q\leq r$, as desired. The other properties are similar and left for the reader.
\end{proof}

\section{$\ETRord$ implies truth predicate for $\mathcal{L}_{\Ord,\Ord}(\in,A)$}\label{Section.ETRord-implies-truth-for-L_Ord,Ord}

In this section, we shall prove the implication $\ref{Main.ETRord}\to\ref{Main.Truth-Ord-Ord}$. The implication $\ref{Main.Truth-Ord-Ord}\to\ref{Main.Truth-Ord-omega}$ is essentially immediate. We begin by providing the natural generalization of the truth predicate concept of definition~\ref{Definition.Truth-predicate-first-order} to the infinitary context as follows.

\begin{definition}\rm\label{Definition.Truth-predicate-infinitary}
If $\kappa,\lambda\leq\Ord$ are cardinals or $\Ord$ itself and $A$ is a fixed class parameter, then a \emph{truth predicate} for the $\mathcal{L}_{\kappa,\lambda}(\in,\hat A)$ language of set theory with a predicate for $A$ is a class $\T$ consisting of pairs $\<\varphi,\vec a>$, where $\varphi$ is a formula in that language and $\vec a$ is a valuation mapping the free variables of $\varphi$ to corresponding set parameters, such that the following recursion is satisfied:
\begin{enumerate}[(a)]
  \item $\T$ judges the truth of $\{=,\in,\hat A\}$-atomic statements correctly:
            $$\T(x=y,\<a,b>)\quad \text{ if and only if }\quad a=b$$
            $$\T(x\in y,\<a,b>)\quad \text{ if and only if } \quad a\in b$$
            $$\T(x\in\hat A,a)\quad \text{ if and only if }\quad a\in A$$
  \item $\T$ performs Boolean logic correctly:
            \begin{align*}
              \T\Bigl(\,\bigwedge_{i\in I}\varphi_i,\vec a\Bigr)\quad & \text{ if and only if }\quad \T(\varphi_i,\vec a)\text{ for all }i\in I\\
              \T(\neg\varphi,\vec a)\quad & \text{ if and only if }\quad \neg\T(\varphi,\vec a)
            \end{align*}
  \item $\T$ performs quantifier logic correctly:
            $$\T(\forall \vec x\, \varphi,\vec a)\text{ if and only if }\forall \vec b\ \T(\varphi,\vec b\concat\vec a),$$
            where $\vec b\concat\vec a$ is the valuation extending $\vec a$ by mapping the variables of $\vec x$ to the objects listed by $\vec b$.
\end{enumerate}
\end{definition}

This generalizes definition~\ref{Definition.Truth-predicate-first-order} from the case of the first-order language \hbox{$\mathcal{L}_{\omega,\omega}(\in,A)$} to the infinitary languages, such as $\mathcal{L}_{\Ord,\omega}(\in,A)$ and $\mathcal{L}_{\Ord,\Ord}(\in,A)$. Let us sublimate a few minor syntactical details, such as the fact that in (b) when referring to $\T(\varphi_i,\vec a)$, we should restrict the valuation $\vec a$ to the free variables of $\varphi_i$, if these were fewer than in $\bigwedge_i\varphi_i$.

\begin{theorem}\label{Theorem.ETRord-implies-Truth-Ord-Ord}
Assume the principle of elementary transfinite recursion $\ETRord$ for recursions of length $\Ord$. Then there is a truth predicate for $\mathcal{L}_{\Ord,\Ord}(\in,A)$, with any class parameter $A$.
\end{theorem}

\begin{proof}
The point is that the existence of a such a truth predicate for this infinitary logic is an elementary transfinite recursion of length $\Ord$, defined by recursion on formulas. The formulas each have an ordinal rank coming from the rank of their parse trees, and we may define the truth of such a formula by reference to the truth of its constituent pieces.
\end{proof}

In light of the results of section~\ref{Section.Separating}, we should not expect to get $\Ord$-iterated truth predicates for the infinitary logic $\mathcal{L}_{\Ord,\Ord}(\in,A)$, or even an $\omega$-iterated truth predicate for  $\mathcal{L}_{\Ord,\omega}(\in,A)$, since those recursions have length $\Ord^2$ or $\Ord\cdot\omega$, and $\ETRord$ is not able to prove that such recursions have a solution. Indeed, it follows from the separation results of that section that the existence of an $\omega$-iterated truth predicate for $\mathcal{L}_{\Ord,\omega}(\in)$ is strictly stronger in consistency strength than $\ETRord$.

\section{Forcing theorem implies truth predicate for $\mathcal{L}_{\Ord,\omega}$}\label{Section.Forcing-theorem-implies-truth-for-L_Ord,omega}

In this section, we shall prove the implication $\ref{Main.atomic}\to\ref{Main.Truth-Ord-omega}$ in the main theorem. In order to do so, let us define a particular class forcing notion $\FA$, from whose atomic forcing relation we shall be able to extract a truth predicate. This is an adaptation and generalization of the forcing defined in~\cite[definition~2.4]{HolyKrapfLuckeNjegomirSchlicht2016:Class-forcing}, and the argument here is based on the analysis of~\cite[section~7]{HolyKrapfLuckeNjegomirSchlicht2016:Class-forcing}. Let $A$ be a proper class parameter. Since the existence of a truth predicate relative to $A$ is invariant under finite changes of $A$, we can assume that $A$ has at least two elements and that $A\neq V$. Let $\Coll(\omega,V)$ be the class partial order having as conditions all finite injective partial functions $f\from\omega\to V$. This is the usual forcing to add a bijection from $\omega$ to $V$. To form $\FA$, we take a disjoint union that augments the forcing $\Coll(\omega,V)$ with additional conditions, by setting
 $$\FA\quad = \quad \Coll(\omega,V)\sqcup \set{e_{n,m}\mid n,m\in\omega}\sqcup \set{a_n\mid n\in\omega}$$
where for $f\in\Coll(\omega,V)$ we define
\begin{equation*}
\begin{split}
   f\leq e_{n,m} &\quad\Iff\quad f(n)\in f(m),\text{ and} \\
   f\leq a_n  & \quad\Iff\quad f(n)\in A.
\end{split}
\end{equation*}
In other words, the condition $e_{n,m}$ is by definition precisely the supremum of the conditions $f$ with $f(n)\in f(m)$, and $a_n$ is by definition the supremum of the conditions $f$ for which $f(n)\in A$. Let us also take $\one=\emptyset\in\Coll(\omega,V)$ to be above all the new conditions $e_{n,m}$ and $a_n$. Recall that $A$ has at least two elements, and that $A\ne V$. Using this, it is easy to check that $\FA$ is separative.

In short, we define the forcing $\FA$ to be basically the collapse forcing $\Coll(\omega,V)$, but augmented with some additional conditions that are the suprema of certain useful classes of conditions. In particular, $\Coll(\omega,V)$ is a dense subclass of $\FA$, and consequently a generic filter $G\of\FA$ will be fully determined by $G\intersect\Coll(\omega,V)$, which will be generic for that forcing. Nevertheless, class forcing differs from set forcing in that a dense subclass of the forcing does not necessarily give rise to the same forcing extensions (see~\cite[section~5]{HolyKrapfSchlicht2017:Characterizations-of-pretameness}). The reason is that the extra conditions, such as $e_{n,m}$ and $a_n$ above, allow one to form names in the larger forcing that are not equivalent to any name in the smaller forcing, even though it is dense. Basically, we augmented the forcing with those extra conditions precisely so that we could use those conditions to form $\FA$-names for objects that we could not have been able to name in the $\Coll(\omega,V)$ forcing alone.

Let us illustrate with a few examples. Define the name
 $$\dot\varepsilon=\set{\<\mathop{\rm op}(\check n,\check m),e_{n,m}>\mid n,m\in\omega}.$$
Notice that $\dot\varepsilon$ is a set-sized name---actually, it is countable---even though it seems to carry information about what will ultimately happen with a proper class of conditions $f$. The reason it can do so is precisely because of the supremum conditions $e_{n,m}$ that we added to $\FA$.

Similarly, we may define the name
 $$\dot A=\set{\<\check n,a_n>\mid n\in\omega}.$$
This is not a name for the class $A$ itself, but rather it is the name for the copy of $A$ on $\omega$ that will be induced by the generic bijection. That is, $\dot A$ is the name for the collection of $n\in\omega$ that will correspond to an element of $A$ by the generic bijection. Again, this is a set-sized name, actually countable, which we can form precisely because of the extra conditions $a_n$ that we had added to the forcing.

Lastly, without using the augmented part of $\FA$, let us define for each set $a$ the name
 $$\dot n_a=\set{\<\check k,(n\mapsto a)>\mid k<n\in\omega},$$
where $(n\mapsto a)=\{\<n,a>\}\in\Coll(\omega,V)$ is the finite partial function defined on only the point $n$ and mapping it to $a$. Thus, $\dot n_a$ is the name of the set of numbers $k$ that are less than the natural number $n$ that will get mapped to $a$ by the generic bijection. In other words, since every natural number is the set of smaller numbers, $\dot n_a$ is the name of the number $n$ that will get mapped to $a$.

The following theorem is a generalization of~\cite[theorem~7.3]{HolyKrapfLuckeNjegomirSchlicht2016:Class-forcing} to the internal \GBC\ treatment of class forcing and also to allow a class parameter.

\begin{theorem}\label{Theorem.Atomic-forcing-relation-to-truth-predicate}
  If the forcing $\FA$, for a fixed class parameter $A$, admits forcing relations for atomic formulas, then there is a truth predicate for $\mathcal{L}_{\Ord,\omega}(\in,A)$.
\end{theorem}

\begin{proof} Fix the class $A$ and suppose that the forcing $\FA$ admits forcing relations for atomic formulas. It follows by theorem~\ref{Theorem.Atomic-to-quantifier-free-forcing-relation} that we have a uniform forcing relation for quantifier-free infinitary assertions in the forcing language $\mathcal{L}_{\Ord,0}(\in,V^{\FA})$.

Our proof will make use of the following purely syntactical translation $\varphi\mapsto\varphi^\star$, where $\varphi\in\mathcal{L}_{\Ord,\omega}(\in,A)$ and $\varphi^\star\in\mathcal{L}_{\Ord,0}(\in,V^{\FA})$. The translation is defined by recursion on $\varphi$ as follows:
\begin{equation*}
 \begin{split}
     ({x\in y})^\star &\quad=\quad x\mathrel{\dot\varepsilon} y \\
     ({x=y})^\star &\quad=\quad x=y \\
     ({x\in A})^\star &\quad=\quad x\mathrel{\dot\varepsilon}\dot A\\
     (\varphi\wedge\psi)^\star &\quad=\quad \varphi^\star\wedge\psi^\star  \\
     (\neg\varphi)^\star &\quad=\quad \neg\varphi^\star \\
     \bigl(\bigwedge_{i\in I}\varphi_i\bigr)^\star &\quad=\quad \bigwedge_{i\in I}\varphi_i^\star  \\
     (\forall x\, \varphi)^\star &\quad=\quad \bigwedge_{m\in\omega} \varphi^\star(\check m).
 \end{split}
\end{equation*}
The interesting cases are $({x \in y})^\star$ and $(\forall x\, \varphi)^\star$. The idea is that the translation is transforming truth assertions about the structure $\<V,\in,A>$ to the corresponding truth assertions about the structure $\langle\omega,\dot\varepsilon,\dot A\rangle$, which intuitively will be made isomorphic by the generic bijection $a\mapsto(\dot n_a)_G$, where $G$ is generic for $\FA$. The point is that $\in$ assertions in $V$ will correspond to $\dot\varepsilon$ relations in the latter structure, and universal assertions $\forall x\, \varphi$ in $V$ will correspond to countable conjunctions in the latter structure, since every object will be placed at some $n$ by the generic bijection, mapping $a$ to $(\dot n_a)_G$. Intuitively, the translation aims at establishing the equivalence
 $$\<V,\in,A>\satisfies\varphi(a)\quad\Iff\quad V[G]\satisfies \bigl(\langle\omega,\dot\varepsilon_G,\dot A_G\rangle\satisfies\varphi^\star((\dot n_a)_G)\bigr).$$
Formalizing this equivalence at the outset, however, is somewhat problematic, because the left-hand side is not expressible unless we already have the desired truth predicate, and the right-hand side makes truth assertions in $V[G]$, which is not a model of any decent set theory, as we have collapsed $V$ to become countable, and so it isn't clear to what extent $V[G]$ has a truth predicate for the structure $\langle\omega,\dot\varepsilon_G,\dot A_G\rangle$. One might hope alternatively for the right hand side to work directly with the structure $\langle\omega,\dot\varepsilon_G,\dot A_G\rangle$, using an externally defined truth predicate; but this will not work properly with non-standard models, since there will be nonstandard formulas $\varphi^\star$ whose truth in that structure will not have a clear meaning in the meta-theory (although this approach does work when the original model is transitive).

Nevertheless, without formalizing the equivalence displayed above, we may instead take it as inspiring an idea that we are able to formalize, specifically, the definition in \GBC\ of a certain predicate $\Tr(\varphi,\vec a)$, which we shall prove is a truth predicate. The idea is similar to that of theorem~\ref{Theorem.Forcing-relation-as-name-for-truth-predicate}.  Namely, if $\varphi$ is an infinitary formula in the language $\mathcal{L}_{\Ord,\omega}(\in,A)$ and $\vec a=\<a_0,\ldots,a_k>$ is a valuation of the (finitely many) free variables of $\varphi$, then we define
  $$\Tr(\varphi,\vec a)\quad\Iff\quad \one\forces_{\FA}\varphi^\star(\dot n_{a_0},\ldots,\dot n_{a_k}).$$
The definition uses the forcing relation in $\FA$ only in the case of the formulas $\varphi^\star$, which are quantifier-free infinitary assertions in the forcing language $\mathcal{L}_{\Ord,0}(\in,V^{\FA})$, and we have a uniform forcing relation for that language by theorem~\ref{Theorem.Atomic-to-quantifier-free-forcing-relation}. Let us prove that this is indeed a truth predicate.

\begin{sublemma}\label{Lemma.p-forces-iff-one-forces}
 For any formula $\varphi\in\mathcal{L}_{\Ord,\omega}(\in,A)$, any sets $a_0,...,a_k$ and any condition $p$,
 $$p\forces\varphi^\star(\dot n_{a_0},\ldots,\dot n_{a_k})\qquad\text{ if and only if }\qquad\one\forces\varphi^\star(\dot n_{a_0},\ldots,\dot n_{a_k}).$$
\end{sublemma}

\begin{proof}
We proceed by induction on the formula $\varphi$. It is important here that the only names appearing as parameters in $\varphi^\star$ have the form $\dot n_a$ for some $a$; for example, the lemma is not true if one allows parameters of the form $\check n$, since if $p(n)$ is defined, then $p$ will force things about $\check n$ that other conditions will not. For the atomic case, it is easy to verify that
\begin{align*}
  p\forces\dot n_a=\dot n_b &\qquad\text{ if and only if }\qquad a=b \\
  p\forces\dot n_a\mathrel{\dot\varepsilon}\dot n_b\ &\qquad\text{ if and only if }\qquad a\in b \\ 
  p\forces\dot n_a\mathrel{\dot\varepsilon}\dot A\ &\qquad\text{ if and only if }\qquad a\in A 
\end{align*}
because we can simply extend $p$ to a condition that decides the exact values of $\dot n_a$ and $\dot n_b$. It follows that if $p\forces \dot n_a=\dot n_b$, $p\forces\dot n_a\mathrel{\dot\varepsilon}\dot n_b$, or $p\forces\dot n_a\mathrel{\dot\varepsilon}\dot A$, then $\one$ also forces that statement, which verifies the atomic case of the lemma. Next, for conjunctions, suppose that $p\forces\bigl(\bigwedge_i\varphi_i(\dot n_a)\bigr)^\star$, which by the definition of the $\star$-translation means $p\forces\varphi_i^\star(\dot n_a)$ for every $i$ (for notational simplicity, we consider just one parameter $\dot n_a$). By the induction hypothesis, this means $\one\forces\varphi_i^\star(\dot n_a)$ for every $i$ and so $\one\forces\bigl(\bigwedge_i\varphi_i(\dot n_a)\bigr)^\star$, as desired. For negation, if $p\forces\neg\varphi^\star(\dot n_a)$, it means there is no $q\leq p$ with $q\forces \varphi^\star(\dot n_a)$. But now, in fact there is no $q$ at all with $q\forces\varphi^\star(\dot n_a)$, for if there were such a $q$, then by induction we would have $\one\forces\varphi^\star(\dot n_a)$, contrary to our assumption that $p\forces\neg\varphi^\star(\dot n_a)$.

Finally, we consider the universal quantifier case. Since universal quantifiers get $\star$-trans\-la\-ted to conjunctions over the natural numbers, let us suppose that $p\forces\bigwedge_{m\in\omega}\varphi^\star(\check m,\dot n_a)$. If $\one$ does not force this conjunction, then there is some condition forcing the negation, and by strengthening further we find a condition $q$ forcing $\neg\varphi^\star(\check m,\dot n_a)$ for some particular $m$. Notice that $\varphi^\star(\check m,\dot n_a)$ does not fall under the induction assumption, since the name $\check m$ is not of the form $\dot n_a$. Nevertheless, by strengthening further if necessary, we may assume $q(m)$ is defined and thus $q\forces\check m=\dot n_b$ for some $b$. It follows by lemma~\ref{Lemma.Forcing-relations-respects-logical-consequence} that $q\forces\neg\varphi^\star(\dot n_b,\dot n_a)$. By the induction hypothesis (and the negation case), it follows now that $\one\forces\neg\varphi^\star(\dot n_b,\dot n_a)$. By strengthening $p$ to a condition $p'$ forcing $\dot n_b=\check m$ for some $m$, this contradicts our assumption that $p\forces\bigwedge_{m\in\omega}\varphi^\star(\check m,\dot n_a)$.
\end{proof}

We now proceed to prove that the class $\Tr$ defined before the lemma is a truth predicate for the language $\mathcal{L}_{\Ord,\omega}(\in,A)$. We need simply to verify the recursive requirements of definition~\ref{Definition.Truth-predicate-infinitary}. For the atomic case, the equivalences
\begin{equation*}
  \begin{split}
     a=b    &\quad\text{ if and only if }\quad \one\forces\dot n_a=\dot n_b\\
     a\in b &\quad\text{ if and only if }\quad \one\forces\dot n_a\mathrel{\dot\varepsilon}\dot n_b \\
     a\in A &\quad\text{ if and only if }\quad \one\forces\dot n_a\mathrel{\dot\varepsilon}\dot A,
  \end{split}
\end{equation*}
follow essentially by design from the definitions of $\dot n_a$, $\dot\varepsilon$ and $\dot A$.

Next, let's check that our definition performs Boolean logic correctly. For simplicity, allow us to consider the case of just one parameter $a$ rather than $\vec a$. In the case of conjunctions, we have $\Tr(\bigwedge_i\varphi_i,a)$ just in case $\one\forces\bigwedge_i\varphi_i^\star(\dot n_a)$, which holds if and only if $\one\forces\varphi_i^\star(\dot n_a)$ for each $i$, which by induction is equivalent to $\Tr(\varphi_i,a)$ for all $i$, as desired.

Consider next negation. If $\Tr(\varphi,a)$ holds, then $\one\forces\varphi^\star(\dot n_a)$, and so it is not the case that $\one\forces\neg\varphi^\star(\dot n_a)$ and so $\Tr(\neg\varphi,a)$ fails. Conversely, if $\Tr(\varphi,a)$ does not hold, then $\one\not\forces\varphi^\star(\dot n_a)$. It follows by the lemma that no condition $p$ can force $\varphi^\star(\dot n_a)$ and consequently $\one\forces\neg\varphi^\star(\dot n_a)$, which implies that $\Tr(\neg\varphi,a)$ holds, thereby fulfilling the desired negation requirement.

Finally, consider the quantifier case. Suppose that $\Tr(\forall x\, \varphi(x),a)$ holds. By definition, this means $\one\forces\bigl(\forall x\, \varphi(x) \bigr)^\star(\dot n_a)$, and by definition of the $\star$-translation, this means $\one\forces\bigwedge_{m\in\omega}\varphi^\star(\check m,\dot n_a)$, which is equivalent to $\one\forces\varphi^\star(\check m,\dot n_a)$ for every $m$. For any set $b$, there is a dense class of conditions $q\leq\one$ with $q\forces\dot n_b=\check m$ for some $m$ and consequently $q\forces\varphi^\star(\dot n_b,\dot n_a)$.  Since the $q$ are dense, this implies $\one\forces\varphi^\star(\dot n_b,\dot n_a)$, and so we have $\Tr(\varphi,b\concat a)$ for every set $b$, as required. Conversely, if $\Tr(\varphi,b\concat a)$ for every set $b$, then $\one\forces\varphi^\star(\dot n_b,\dot n_a)$ for every $b$, and from this it follows that $\one\forces\varphi(\check m,\dot n_a)$ for any particular $m$, since it is dense to force $\check m=\dot n_b$ for some $b$. Thus, $\one\forces\bigwedge_{m\in\omega}\varphi^\star(\check m,\dot n_a)$ and consequently $\Tr(\forall x\,\varphi,a)$, as desired.
\end{proof}

\section{Truth predicate for $\mathcal{L}_{\Ord,\omega}(\in,A)$ implies $\Ord$-iterated truth predicate for $\mathcal{L}_{\omega,\omega}(\in,A)$}\label{Section.Truth-predicate-for-L_{Ord,omega}-implies-Ord-iterated-truth}

In this section, we prove the implication $\ref{Main.Truth-Ord-omega}\to\ref{Main.Iterated-truth}$. To begin, let us define what it means to have an $\Ord$-iterated truth predicate. The idea is to have a truth predicate that applies not only to statements in the language of set theory, but to statements about (earlier stages of) truth in that language. So the $\Ord$-iterated truth predicate will make truth assertions in ordinal stages.

\begin{definition}\label{Definition.Ord-iterated-truth-first-order}\rm
An $\Ord$-iterated truth predicate for first-order truth, with a class parameter $A$, is a class $\Tr$ consisting of triples $\<\beta,\varphi,\vec a>$, where $\beta$ is an ordinal, $\varphi$ is a formula in the first-order language of set theory augmented with a predicate for $A$ and also with a trinary predicate symbol $\hat\Tr$ to be used for (iterated) truth assertions, and $\vec a$ is a valuation mapping the free variables of $\varphi$ to corresponding parameters, such that the following recursion is satisfied:
\begin{enumerate}[(a)]
  \item $\Tr$ judges the truth of $\{=,\in,\hat A\}$-atomic statements correctly:
          \begin{align*}
            \Tr(\beta,x=y,\<a,b>)\quad    &\text{ if and only if }\quad a=b\\
            \Tr(\beta,x\in y,\<a,b>)\quad &\text{ if and only if } \quad a\in b\\
            \Tr(\beta,x\in\hat A,a)\qquad  &\text{ if and only if }\quad a\in A
          \end{align*}
  \item $\Tr$ judges atomic assertions of the truth predicate self-coherently:
            \begin{align*}
              &\Tr(\beta,\hat\Tr(x,y,z),\<\alpha,\varphi,\vec a>)\\
              &\qquad\text{ if and only if }\quad \alpha<\beta\text{ and }\Tr(\alpha,\varphi,\vec a)
            \end{align*}
  \item $\Tr$ performs Boolean logic correctly:
            \begin{align*}
              \Tr(\beta,\varphi\wedge\psi,\vec a)\quad & \text{ if and only if }\quad \Tr(\beta,\varphi,\vec a)\text{ and }\Tr(\beta,\psi,\vec a)\\
              \Tr(\beta,\neg\varphi,\vec a)\quad & \text{ if and only if }\quad \neg\Tr(\beta,\varphi,\vec a)
            \end{align*}
  \item $\Tr$ performs quantifier logic correctly:
            $$\Tr(\beta,\forall x\, \varphi,\vec a)\quad\text{ if and only if }\quad\forall b\, \Tr(\beta,\varphi,b\concat\vec a)$$
\end{enumerate}
\end{definition}

Note the crucial clause (b), which insists that atomic truth assertions made at stage $\beta$ can refer only to earlier stages of truth, $\alpha<\beta$. When the formula $\varphi$ has no free variables, then to improve readability we shall write $\Tr(\beta,\varphi)$ in place of $\Tr(\beta,\varphi,\<>)$, since in this case the valuation is empty.

When $\Tr(\beta,\varphi,\vec a)$, then we shall say that $\varphi[\vec a]$ is declared true by the predicate at stage $\beta$. In this way, we can extract from the uniform truth predicate the sequence of individual truth predicates $\Tr_\beta(\varphi,\vec a)$ for the truth assertions made at each stage $\beta$. These predicates perform as one would want for the non-uniform manner of iterated truth, where one makes truth assertions always only at a particular stage of truth, with a separate predicate symbol for each stage. We should like to emphasize, however, that the uniform truth predicate is stronger than this, because it allows formulas to quantify over the stages of truth. For example, the uniform iterated truth predicate allows one to express various liar-type sentences, which have an interesting nature with respect to the iterated truth predicate.

Consider the following instance, a sentence $\sigma$ whose truth value will systematically and endlessly alternate between true and false at successive ordinal stages. Specifically, the sentence $\sigma$ expresses that there is no immediately preceding stage at which $\sigma$ is true. To formalize this in the language of the iterated truth predicate, let $\top$ be any tautologically valid sentence, such as $\forall x\, x=x$, and notice that in light of requirement (b) above, the assertion $\hat\Tr(\alpha,\top)$ is judged true at stage $\beta$ if and only if $\alpha<\beta$; we can use this feature to quantify in effect over the `earlier' stages of truth. By the \Godel-Carnap fixed-point theorem, which asserts of every formula $\varphi(x)$ that there is a sentence $\sigma$ such that our theory proves $\sigma\iff\varphi(\sigma)$, it follows that there is a sentence $\sigma$ for which:
    $$\sigma\longleftrightarrow\neg\exists\alpha\,\bigl[\Tr(\alpha,\sigma)\wedge \Tr(\alpha,\top)\wedge\neg\Tr(\alpha+1,\top)\bigr],$$
and furthermore, this equivalence is valid at every stage of truth. Said plainly, $\sigma$ asserts that it isn't the case that $\sigma$ is true at a stage which is the largest `previous' stage. At stage $0$, there is no largest previous stage, and so indeed $\sigma$ is true at stage $0$. Thus, it will become false at stage $1$, and therefore true again at stage $2$. Whenever $\sigma$ is true at stage $\beta$, then it will become false at stage $\beta+1$ and true again at stage $\beta+2$. The sentence $\sigma$ will be true at limit ordinal stages, since there is no largest previous stage. So $\sigma$ flips between true and false for the iterated truth predicate, being true at every even ordinal stage and false at every odd ordinal stage, even though it is the very same sentence being considered each time.

Similar constructions yield sentences $\sigma$ that are true exactly at the stages in a class $A$ or exactly at certain other stages in a way that is convenient. Such kind of sentences are not generally possible in the weaker language having only truth predicates $\Tr_\beta$ for each stage of truth separately, without the possibility of quantifying over the index $\beta$, since in that language a sentence $\varphi$ would refer to only finitely many of those predicates, and therefore the truth or falsity of the sentence would stabilize at stages of truth beyond those explicitly mentioned in $\varphi$.

\begin{theorem}\label{Theorem.Truth-Ord-omega-implies-Iterated-truth}
For any class $A$, if there is a truth predicate for the infinitary language $\mathcal{L}_{\Ord,\omega}(\in,A)$, then there is an $\Ord$-iterated truth predicate in the first-order language $\mathcal{L}_{\omega,\omega}(\in,\hat\Tr,A)$.
\end{theorem}

This will establish implication $\ref{Main.Truth-Ord-omega}\to\ref{Main.Iterated-truth}$ in the main theorem.

\begin{proof}
Suppose that $T(\psi,\vec a)$ is a (non-iterated) truth predicate for the infinitary language $\mathcal{L}_{\Ord,\omega}(\in,A)$, where $A$ is a fixed class parameter. Note that in light of lemma~\ref{Lemma.Every-set-definable}, we don't actually need the parameters $\vec a$, since every set is definable, and so the entire semantic content of the truth predicate is actually contained in its sentences. Nevertheless, we shall carry on with the parameters $\vec a$.

Using the truth predicate $\T$, we shall define another truth predicate, an $\Ord$-iterated truth predicate $\Tr(\beta,\varphi,\vec a)$ for first-order assertions $\varphi$ in the language $\mathcal{L}_{\omega,\omega}(\in,\hat\Tr,A)$. In order to do so, we first define a certain syntactic translation
    $$(\beta,\varphi)\mapsto \varphi^*_\beta,$$
where $\beta$ is an ordinal and $\varphi$ is a formula in the first-order language with an iterated truth predicate \hbox{$\mathcal{L}_{\omega,\omega}(\in,\hat\Tr,A)$}. The resulting formula $\varphi^*_\beta$ is an assertion in \hbox{$\mathcal{L}_{\Ord,\omega}(\in,A)$}, without any truth predicate. This translation is defined by induction on $\beta$ and $\varphi$, as follows:

\begin{enumerate}[(a)]
  \item Atomic formulas not mentioning the truth predicate are not changed by the translation:
   $$(x\in y)^*_\beta\quad=\quad(x\in y),$$
   $$(x=y)^*_\beta\quad=\quad(x=y),$$
   $$(x\in\hat A)^*_\beta\quad=\quad(x\in\hat A).$$
  \item $\hat\Tr(x,y,z)^*_\beta$ is the assertion $$\bigvee_{\xi<\beta\atop \psi\in\mathcal{L}_{\omega,\omega}(\in,\hat\Tr,\hat A)}\hskip-1.5em\bigl[``x=\xi"\wedge ``y=\psi"\wedge\exists\vec a\ \text{valuation}_\psi(z,\vec a)\wedge\psi^*_\xi(\vec a)\bigr].$$
  \item $(\varphi\wedge\psi)^*_\beta=\varphi^*_\beta\wedge\psi^*_\beta$.
  \item $(\forall x\, \varphi)^*_\beta=\forall x\, \varphi^*_\beta$.
\end{enumerate}

The key translation occurs in step (b), replacing atomic instances of the truth predicate with certain infinitary formulas. The formula $\text{valuation}_\psi(z,\vec a)$ asserts that $z$ is a valuation mapping the variables that happen to be free in $\psi$ to the objects in $\vec a$. Basically, the stage $\beta$ translation of the atomic truth assertion $\hat\Tr(x,y,z)$ is the assertion that (i) $x$ is some stage $\xi$ less than $\beta$; (ii) $y$ is some formula $\psi$; and (iii) $z$ is a valuation of the free variables of that formula to objects $\vec a$ for which $\psi^*_\xi(\vec a)$ holds. In particular, in statement (b) we used the expressions $``x=\xi"$ and $``y=\psi"$, and by this we mean the formulas $\theta_\xi(x)$ and $\theta_\psi(y)$ provided by lemma~\ref{Lemma.Every-set-definable}, which as we mentioned does not involve using $\xi$ or $\psi$ (or any code for $\psi$) as a parameter. We should also like to emphasize that we do not need $\ETRord$ in order to make this recursive definition of the translation, since it is an ordinary transfinite recursion of length $\Ord$ on sets, not a class recursion.

We now define our proposed iterated truth predicate $\Tr(\beta,\varphi,\vec a)$ to hold if and only if $\T(\varphi^*_\beta,\vec a)$. We claim that this relation fulfills the requirements to be an iterated truth predicate.

Because $\T$ is a truth predicate for assertions in $\mathcal{L}_{\Ord,\omega}(\in,A)$, it follows easily that our iterated truth predicate $\Tr$ works correctly on $\{\in,=,\hat A\}$-atomic formulas and on Boolean combinations and quantifiers. The only difficult part is to verify that we have judged the atomic truth assertions themselves in a self-coherent manner. What we need to prove is that $\Tr(\beta,\hat\Tr(x,y,z),\<\alpha,\varphi,\vec a>)$ holds if and only if $\alpha<\beta$ and $\Tr(\alpha,\varphi,\vec a)$.

To see this, notice that $\Tr(\beta,\hat\Tr(x,y,z),\<\alpha,\varphi,\vec a>)$ holds, by the definition of $\Tr$, just in case $\T(\hat\Tr(x,y,z)^*_\beta,\<\alpha,\varphi,\vec a>)$ holds, which is equivalent to
    $$\T\bigl(\hskip-1.5em\bigvee_{\xi<\beta\atop \psi\in\mathcal{L}_{\omega,\omega}(\in,\hat\Tr,\hat A)}\hskip-1.5em\bigl(``x=\xi"\wedge ``y=\psi"\wedge\psi^*_\xi\bigr)\, ,\, \<\alpha,\varphi,\vec a>\bigr).$$
Since $\T$ is a truth predicate, we may unwrap the meaning of this disjunction to see that this holds if and only if $\alpha<\beta$ and $\T(\varphi^*_\alpha,\vec a)$, since the disjuncts can be realized only with $\alpha=\xi$ and $\varphi=\psi$ themselves. By the definition of $\Tr$, this is equivalent to $\Tr(\alpha,\varphi,\vec a)$. So we have verified that $\Tr$ is an iterated truth predicate for first-order truth with the class parameter $A$, as desired.
\end{proof}

One can extract from this argument a corresponding result for having merely a $\kappa$-iterated truth predicate, given a truth predicate for the $\mathcal{L}_{\kappa,\omega}$ language of set theory.

\begin{theorem}
  For any class $A$, if there is a truth predicate for the infinitary language $\mathcal{L}_{\kappa,\omega}(\in,\hat A)$, where $\kappa$ is any uncountable cardinal, then there is an $\kappa$-iterated truth predicate in the first-order language $\mathcal{L}_{\omega,\omega}(\in,\hat\Tr,\hat A)$.
\end{theorem}

\begin{proof}
The point is that to define $\Tr(\beta,\varphi,\vec a)$ in the previous theorem, we took a disjunction of size $\omega\cdot\beta$, which will still be less than $\kappa$ for $\beta<\kappa$, if $\kappa$ is an uncountable cardinal. So exactly the same translation and definition of $\Tr$ works up to $\kappa$.
\end{proof}

\section{Iterated truth predicates imply the elementary transfinite recursion principles}\label{Section.Ord-iterated-truth-implies-ETRord}

In this section, we shall prove implication $\ref{Main.Iterated-truth}\to\ref{Main.ETRord}$ in the main theorem, and indeed, we shall prove $\ref{Main.Iterated-truth}\iff\ref{Main.ETRord}$. To do so, we shall undertake a refinement of the following theorem of Fujimoto~\cite{Fujimoto2012:Classes-and-truths-in-set-theory} (see also Gitman and Hamkins~\cite{GitmanHamkins2016:OpenDeterminacyForClassGames}).

\begin{theorem}
  The principle of elementary transfinite recursion $\ETR$ is equivalent over \GBC\ to the existence of iterated truth predicates along any well-founded class relation.
\end{theorem}

The forward implication of this is straightforward, as the truth predicate itself is defined as the solution to a certain recursion. The real content of the theorem is the converse, that from any sufficiently iterated truth predicate one can extract a solution of a given recursion. For the full \ETR, one must consider recursions of length exceeding $\Ord$, as well as iterated truth predicates of length strictly longer than $\Ord$, for any well-founded class relation.

What we want to prove here is that the equivalence goes through when one restricts the recursions and iterations to length $\Ord$, or indeed, to length $\Gamma$ for any well-ordered infinite class $\Gamma$.

\begin{theorem}\label{Theorem.ETR_Gamma-iff-Gamma-iterated-truth-predicates}
For any class well-order $\Gamma$, with $\omega^\omega\leq\Gamma$, the principle of elementary transfinite recursion $\ETR_\Gamma$ for recursions of length $\Gamma$ is equivalent over \GBC\ to the existence of $\Gamma$-iterated truth predicates, allowing any class parameter in each case.
\end{theorem}

This will establish the equivalence $\ref{Main.Iterated-truth}\iff\ref{Main.ETRord}$ in the main theorem, if we consider the case $\Gamma=\Ord$. The expression $\omega^\omega$ refers to the countable ordinal arising via ordinal exponentiation of $\omega$ with itself.

\begin{proof}
We follow the proof of~\cite[theorem 8]{GitmanHamkins2016:OpenDeterminacyForClassGames}. Let us emphasize that both statements in the theorem make their assertions universally for all class parameters.

The forward implication is basically straightforward, since the iterated truth predicate itself is defined by a transfinite recursion of length $\omega\cdot\Gamma$, by recursion on formulas. Namely, if we have the predicate $\Tr\restrict\alpha$ up to stage $\alpha\in\Gamma$, then we can define $\Tr$ at stage $\alpha$ by a length $\omega$-recursion on formulas, by reference to the partial solution $\Tr\restrict\alpha$. So the entire recursion has length $\omega\cdot\Gamma$. Note that $\ETR_\Gamma$ implies $\ETR_{\Gamma+\Gamma}$, since one need only perform the recursion up to stage $\Gamma$, and then define a new recursion for the rest of the way, using the partial solution as a new parameter. And since $\omega^\omega\leq\Gamma$, it follows that $\omega\cdot\Gamma<\Gamma+\Gamma$ (use that $\Gamma=\omega^\omega\cdot\Lambda+\alpha$ for some $\Lambda$ and some $\alpha<\omega^\omega$, and observe $\omega\cdot\Gamma=\omega\cdot(\omega^\omega\cdot\Lambda+\alpha)=
\omega^\omega\cdot\Lambda+\omega\cdot\alpha=\Gamma+\omega\cdot\alpha<\Gamma+\Gamma$). Thus, we can get the $\Gamma$-iterated truth predicate from $\ETR_\Gamma$, as desired.

Conversely, suppose that for a class parameter $A$, we have an iterated truth predicate $\Tr$ of length $\Gamma$ for first-order truth relative to the parameter $A$. Now suppose that we have an instance of $\ETR_\Gamma$, iterating a formula $\varphi(x,\alpha,A,X)$, where we seek a solution $S$ up to $\Gamma$, a class $S\of\Gamma\times V$ for which $S_\alpha=\set{x\mid \varphi(x,\alpha,A,S\restrict\alpha)}$ for every $\alpha<\Gamma$. We claim that using the truth predicate as a class parameter, we may define such a solution $S$. To do this, we claim first that there is a formula $\bar\varphi$ such that if one extracts from $\Tr$ the class defined by $\bar\varphi$, namely, $S=\set{\<\alpha,x>\mid \Tr(\alpha,\bar\varphi,x)}$, then $S$ is a solution to the recursion of $\varphi$ along $\Gamma$. The formula $\bar\varphi$ should simply be chosen so that $\<V,{\in},A,\Tr\restrict \alpha>\satisfies\bar\varphi(x,\alpha)$ if and only if $\<V,{\in},A,S\restrict \alpha>\satisfies\varphi(x,\alpha)$, where $S$ is defined as just mentioned using $\bar\varphi$. Such a formula $\bar\varphi$ exists by the \Godel-Carnap fixed-point lemma: for any $e$, let $\psi(e,x,\alpha)$ be the assertion $\<V,{\in},A,\set{\<\alpha,x>\mid \Tr(\alpha,e,x)}>\satisfies\varphi(x,\alpha)$, and then by the usual fixed-point trick find a formula $\bar\varphi(x,\alpha)$, for which $\<V,{\in},A,\Tr\restrict \alpha>\satisfies\psi(\bar\varphi,x,\alpha)\iff\bar\varphi(x,\alpha)$. It follows that the class $S$ iteratively defined from $\Tr$ by $\bar\varphi$ satisfies $\varphi$ at each step and therefore is a solution to the recursion of $\varphi$ up to $\Gamma$, as desired.
\end{proof}

An essentially similar idea works with infinitary formulas, provided that $\Ord\cdot\Gamma\leq\Gamma+\Gamma$, which is to say that $\Gamma$ is at least $\Ord^\omega$, with ordinal exponentiation.

\begin{theorem}
 For any infinite class well-order $\Gamma$, of order type at least $\Ord^\omega$, the principle $\ETR_{\Gamma}\bigl(\mathcal{L}_{\Ord,\Ord}(\in,A)\bigr)$ of elementary transfinite recursion for recursions of length $\Gamma$ in the infinitary language with a class parameter $A$ is equivalent over \GBC\ to the existence of $\Gamma$-iterated truth predicates for that language.
\end{theorem}

\begin{proof}
The forward implication is proved by observing that the iterated truth predicate is precisely a solution of a recursion. The converse implication is proved as in the previous theorem by defining a solution of the recursion by reference to the iterated truth predicate.
\end{proof}

\section{Class-join separation}\label{Section.Class-join-separation}

In this section, we shall prove the implications $\ref{Main.ETRord}\to\ref{Main.Class-join-separation}\to\ref{Main.Truth-Ord-omega}$ in the main theorem. To begin, we define the principle of \emph{$\ETRord$-foundation} to be the assertion that every instance of elementary transfinite recursion of length $\Ord$ either has a solution, or else fails at some least stage $\alpha\leq\Ord$. That is, either there is a solution or there is $\alpha\leq\Ord$ such that for every $\beta<\alpha$ there is a solution of the recursion of length $\beta$, but there is no solution of length $\alpha$. Since the partial solutions are unique when they exist, perhaps some readers expect that one could simply combine those earlier solutions into one uniform solution; but such an argument would appeal to a class-replacement principle that is not provable in \GBC. (Consider the difficulty, for example, of combining the various $\Sigma_n$-truth predicates into a uniform truth predicate, if the model is $\omega$-standard and has only definable classes.) Meanwhile, the principle of $\ETRord$-foundation is a consequence of the principle of \emph{$\Pi^1_1$-foundation}, which asserts that every $\Pi^1_1$-definable class of ordinals has a least element, and both of these principles are true in any well-founded model of \GBC. That is, in transitive models, we get it for free.

Let us now define a separation-like principle that we call the \emph{class-join separation principle}. This is the assertion that for any class $\Phi$ of $\mathcal{L}_{\Ord,\omega}(\in,A)$-formulas in finitely many free variables, where $A$ is any class parameter, if every formula $\varphi\in\Phi$ admits a truth predicate $T_\varphi$, then $\set{\vec a\mid \exists\varphi\in\Phi\ T_\varphi(\varphi,\vec a)}$ exists as a class. 
(By a truth predicate $T_\varphi$, we mean a class that satisfies the conditions in Definition \ref{Definition.Truth-predicate-first-order}, but only for subformulas of $\varphi$.) 
The idea is that this class is essentially what we would want to mean by the class $\set{\vec a\mid \bigvee_{\varphi\in\Phi}\varphi(\vec a)}$, defined by a class-sized join. The principle asserts that such class-sized joins can be used to define classes, even when one lacks a uniform truth predicate and only has truth predicates for each formula $\varphi$ individually. Indeed, if there is a uniform truth predicate application to all $\varphi\in\Phi$, then the instance of the class-join separation principle would follow from the ordinary separation axiom of \GBC\ simply by using that predicate as a parameter. The principle, instead, is about unifying a class-indexed collection of separate truth predicates.

The dual principle, the \emph{class-meet separation principle}, asserts of every $\Phi$ as above that $\set{\vec a\mid \forall\varphi\in\Phi\ T_\varphi(\varphi,\vec a)}$ exists. In other words, the principle allows us to use the class-sized conjuction $\bigwedge_{\varphi\in\Phi}\varphi(\vec x)$ to define a class, provided that we have truth predicates for each individual $\varphi\in\Phi$.

It will be convenient to have the following folklore lemma, showing that every set is definable by a suitable infinitary formula.

\begin{lemma}\label{Lemma.Every-set-definable}
 For every set $a$, there is a formula $\theta_a(x)$ in the infinitary language $\mathcal{L}_{\Ord,\omega}(\in)$ that defines $a$. More specifically, there is a truth predicate for the class of these formulas $\theta_a$ and their subformulas, and with respect to this truth predicate, $\theta_a(x)$ is true only when $x=a$.
\end{lemma}

In particular, in any transitive model of set theory, $\theta_a$ defines $a$, meaning that $M\satisfies\theta_a[b]$ just in case $b=a$.

\begin{proof}
 We define the formulas $\theta_a$ by the following $\in$-recursion:
 $$\theta_a(x)\quad=\quad\forall z\bigl[z\in x \iff\bigvee_{u\in a} \theta_u(z)\bigr].$$ The formula asserts that the elements of $x$ are precisely the objects satisfying the definition of an element of $a$. In \GBC\ we can define a truth predicate for the class of $\theta_a(x)$ and their subformulas simply by extending the class $\{\<\theta_a(x),a>\mid a\in V\}$ in the natural way to the subformulas of the $\theta_a$, which are these, joins of these and the formulas $z\in x\iff\bigvee_{u\in a}\theta_u(z)$. Basically, since we know that we want $\theta_a(x)$ to define $x=a$, we can use that to build the truth predicate. One can verify inductively that this is indeed a truth predicate for these formulas.
\end{proof}

Subsequently, we shall write simply $``x=a"$ for the formula $\theta_a(x)$, with the understanding that $a$ is not appearing here as a parameter, but rather the hereditary $\in$-structure of $a$ appears essentially in the parse tree of the formula itself.

\begin{theorem}
  The following are equivalent over \GBC:
  \begin{enumerate}[\rm(i)]
    \item The principle $\ETRord$.
    \item For every class $A$, there is a truth predicate for $\mathcal{L}_{\Ord,\omega}(\in,A)$.
    \item The class-join separation principle plus $\ETRord$-foundation.
    \item The class-meet separation principle plus $\ETRord$-foundation.
  \end{enumerate}
\end{theorem}

\begin{proof}
(i)$\iff$(ii) This was established by the results of sections~\ref{Section.ETRord-implies-truth-for-L_Ord,Ord} and~\ref{Section.Ord-iterated-truth-implies-ETRord}.

(ii)$\to$(iii) The principle of $\ETRord$-foundation is an immediate consequence of $\ETRord$. And once we have a uniform $\mathcal{L}_{\Ord,\omega}(\in,A)$-truth predicate $\T$, any instance of class-meet separation with predicate $A$ reduces to an ordinary instance of separation relative to the class predicate $\T$. 

(iii)$\iff$(iv) This follows easily from the de Morgan law.

(iv)$\to$(ii) Assume the class-join separation principle plus $\ETRord$-foundation. Consider the recursive definition of a truth predicate for $\mathcal{L}_{\Ord,\omega}(\in,A)$. This is an $\Ord$-recursion, defined by recursion on the rank of the parse trees of the formulas. If the recursion has a solution, then we have the desired truth predicate. If it does not, then by $\ETRord$-foundation, the recursion fails at some least rank $\alpha\leq\Ord$. So for every $\beta<\alpha$, we have a uniform truth predicate $\T_\beta$ for the class $\mathcal{L}^\beta$ consisting of formulas whose parse tree has rank less than $\beta$, but there is no such uniform truth predicate covering all formulas of rank less than $\alpha$. Since we can easily extend a truth predicate for formulas of rank $\beta$ to rank $\beta+1$, it follows that $\alpha$ must be a limit ordinal or $\Ord$ itself, and so the class of formulas of rank less than $\alpha$ is the union of those of rank less than some $\beta<\alpha$. Let $\Phi$ be the class of formulas in $\mathcal{L}^\alpha$ with free variable $x$. Using the class-join separation principle, we may define a predicate $\T_\alpha$ as follows:
 $$\T_\alpha\quad=\quad\set{(y,z)\mid \bigvee_{\psi\in\Phi}\bigl(``y=\psi"\wedge\psi(z)\bigr)}.$$
What we mean by the join is the assertion $\exists\psi\in\Phi\, \T_\psi\bigl(``y=\psi"\wedge\psi(z)\bigr)$, where $``y=\psi"$ is the formula $\theta_\psi(y)$ of lemma~\ref{Lemma.Every-set-definable} and where $\T_\psi$ is a truth predicate for the conjunction $``y=\psi"\wedge\psi(z)$. Indeed we have such a truth predicate, because we can simply combine the truth predicate of $\psi$ arising from the fact that it has rank less than $\alpha$ with the truth predicates provided by lemma~\ref{Lemma.Every-set-definable}.

Thus, we have included $(\varphi,a)$ in $\T_\alpha$ just in case $\varphi$ is some formula $\psi$ of rank less than $\alpha$ for which $\psi(z)$ is true. It is now easy to verify that this is indeed a uniform truth predicate for all formulas of rank less than $\alpha$, contrary to the assumption that there was no such solution at stage $\alpha$.
\end{proof}

At bottom, the argument is that $\ETRord$-foundation tells you that if the recursion fails, then it does so at a particular stage, and class-join separation allows you to unify the earlier-stage truth predicates into a uniform predicate, showing that that stage was not a failure after all. So the recursion must succeed.

\section{Clopen determinacy for games of rank at most $\Ord+1$}

We shall now prove the implications $\ref{Main.ETRord}\to\ref{Main.Clopen-determinacy}\to\ref{Main.Truth-Ord-omega}$, using arguments that amount to a refinement of corresponding results of Gitman and Hamkins in~\cite{GitmanHamkins2016:OpenDeterminacyForClassGames}.

A clopen class game is played on a well-founded class tree $T\of V^{<\omega}$, whose terminal nodes are labeled as a win for one or the other player. The game starts at the root node, which we place at the top, so the play proceeds downward. The players take turns, each subsequently selecting a child node of the current node and then continuing in turn from this node. The game ends when a terminal node is reached---by well-foundedness this must happen at some finite stage---and the winner is determined by the label of that node.

A well-founded class tree $T$ admits a ranking function with a well-ordered class relation $\Gamma$, if there is a labeling of the nodes of the tree with elements of that relation, in such a way that every child node has lower rank than its parent. A tree has rank at most $\Ord+1$, therefore, if we can label the root of the tree with $\Ord$ and all other nodes of the tree with ordinals, in such a way that these ordinals descend as one moves down in the tree. A ranking is \emph{continuous}, if it obeys the recursive property that the rank of any node $p$ is exactly the supremum of $\rank(q)+1$ in $\Gamma$ for every child $q$ of $p$. The principle $\ETR_\Gamma$ implies that every tree with a $\Gamma$ ranking function has a continuous $\Gamma$ ranking function, but it isn't clear whether one can prove this without an appeal to some fragment of \ETR.

Meanwhile, we claim that the question of whether or not a well-founded class tree $T$ has rank at most $\Ord+1$ is actually a first-order-expressible property of the tree, and furthermore in \GBC\ such trees always admit first-order-definable continuous ranking functions, without requiring any appeal to a fragment of \ETR. To see this, notice that if we consider $T\intersect V_\theta$, which is a well-founded set-sized tree, then it has an continuous ranking function, since \ZFC\ proves that every well-founded set relation has a continuous ranking function. As $\theta$ increases, the rank of any fixed node in $T\intersect V_\theta$ never decreases. If every non-root node in $T$ has the property that its ranks in these approximation trees $T\intersect V_\theta$ eventually stabilize for large enough $\theta$, then in fact those limit values form an acceptable continuous ranking of the whole tree $T$, if we should place label $\Ord$ or a suitable ordinal on the root node. And conversely, if we are able to rank the whole tree, then those ranks also serve as ranks in the trees $T\intersect V_\theta$. So a tree has rank at most $\Ord+1$ if and only if every non-root node has an eventually stabilizing ordinal rank in $T\intersect V_\theta$, which is a first-order property about the tree. Since the assignment to the nodes of the corresponding limit rank value is first-order definable, we thereby achieve a definable continuous ranking of such trees, as we claimed.

The existence of rankings in a game tree is connected with the requirement in some games that a particular player must count down in a well-order during play. This counting-down feature in a game is often a convenient way to ensure that the game is clopen, since the player cannot count down forever, and the outcome of the game is known when the clock runs out. If a game tree has one player counting down in $\Gamma$ during play, then we can rank the tree with elements of $2\cdot\Gamma+1$, since we label the root node with $2\cdot\Gamma$, and then whenever the count-down player has just announced $\alpha$, we label that node with $2\cdot\alpha$, and if it is the other players turn, we label with $2\cdot\alpha+1$. Thus, using the fact that $2\cdot\Ord=\Ord$, a game where one of the players must count down in the ordinals has rank at most $\Ord+1$.

\begin{definition}\rm
 The principle of \emph{determinacy} for clopen class games of rank at most $\Gamma$, a class well-order, is the assertion that for every clopen class game with a game tree of rank at most $\Gamma$, one of the players has a winning strategy.
\end{definition}

We emphasize that we are referring here to the ordinal rank of the game tree, which is not the same as the game values that would arise in the open determinacy analysis of the game. For example, the game tree can have a very high rank, even if the first player has a winning move on the first move, which would make the game value very low. Meanwhile, the game value for a clopen game, when it exists (it is defined by a class recursion and in \GBC\ there needn't be a solution of that recursion) is bounded by the rank of the tree.

A winning strategy for a player is a function on the game tree, which selects of every parent node a child of it, in such a way that every play of the game that conforms with those choices on that player's moves, leads to a terminal node that is a win for that player.

\begin{theorem}\label{Theorem.ETRord-implies-clopen-determinacy}
  The principle of elementary transfinite recursion $\ETRord$ for $\Ord$-length recursions implies the principle of determinacy for clopen class games of rank at most $\Ord+1$. More generally, for any class well-order $\Gamma$, the principle $\ETR_\Gamma$ for recursions of length $\Gamma$ implies determinacy for clopen class games of rank at most $\Gamma+1$.
\end{theorem}

This will establish implication $\ref{Main.ETRord}\to\ref{Main.Clopen-determinacy}$ in the main theorem.

\begin{proof}
Using the back-propagation method, due originally to Zermelo in his proof of the fundamental theorem of finite games, we shall label every node in the game tree $T$ as a win either for player I or for player II, and these designations will provide a winning strategy for whoever gets their label on the root node, the strategy being: stay on the nodes with your label. To begin, assume $\ETR_\Gamma$ for a class well-order $\Gamma$, and suppose $T$ is a well-founded game tree $T$ of rank at most $\Gamma+1$. Thus, the root node in the tree gets rank $\Gamma$, but all other nodes will have a rank below $\Gamma$. Using $\ETR_\Gamma$, we apply the back-propagation method to label the nodes of the tree with player I or player II, by recursion on rank. The terminal nodes, with rank $0$, are already labeled for us. If all children of a node $t$ are labeled, then if it is player I's turn to play and there is a child node labeled I, then we place label I on $t$, and otherwise II; similarly, if it is player II's turn to play, and there is a child node labeled II, then we place label II on $t$, and otherwise I. By the principle $\ETR_\Gamma$, this labels all the non-root nodes of the tree. We may now place the corresponding label on the root, labeled following the same back-propagation rule. From this labeling, we can get a winning strategy: whoever has their label on the root node can always stay on their own labels, and thereby win the game.
\end{proof}

Next, we establish the implication $\ref{Main.Clopen-determinacy}\to\ref{Main.Truth-Ord-omega}$ in the main theorem.

\begin{theorem}\label{Theorem.Clopen-determinacy-implies-truth}
 The principle of determinacy for clopen class games of rank at most $\Ord+1$ implies the existence of a truth predicate for $\mathcal{L}_{\Ord,\omega}(\in,A)$ for any class parameter $A$.
\end{theorem}

\begin{proof}
We follow the main ideas of Gitman and Hamkins~\cite[theorem~9]{GitmanHamkins2016:OpenDeterminacyForClassGames}, using a natural infinitary analogue of the truth-telling game, where the interrogator counts down in the ordinals. Specifically, consider the truth-telling game for assertions in the logic $\mathcal{L}_{\Ord,\omega}(\in,A)$. There are two players, the \emph{interrogator} and the \emph{truth-teller}. At each move, the interrogator issues a challenge in the form of a set $\Phi$ of infinitary formulas $\varphi\in\mathcal{L}_{\Ord,\omega}(\in,A)$ and a valuation $\vec a$ of their free variables. The interrogator must also state an ordinal $\alpha$, which will strictly descend during play; we call this the count-down clock. The truth-teller replies to the inquiry by stating of each formula $\varphi\in\Phi$ whether it is true or false at $\vec a$ (not necessarily truthfully). If the truth-teller happens to declare an existential formula $\exists x\,\psi(x)$ to be true, then she is also obligated to provide a witness $b$ and declare that $(\psi,b\concat\vec a)$ is true. The truth-teller loses, if she should ever explicitly violate the Tarskian recursion. The interrogator loses if the clock runs out (which must happen eventually, in finitely many moves).

This game has rank at most $\Ord+1$, because of the ordinal count-down clock. Now, we simply argue as in~\cite{GitmanHamkins2016:OpenDeterminacyForClassGames} that if the truth-teller has a winning strategy, then the truth assertions made by that strategy will be independent of the play, for all plays in which there remains sufficient time on the clock. This can be proved by induction on formulas. It is clearly true for atomic formulas. And if it is true for a formula, it will be true for the negation, since any violation of this can be transformed into a violation of the Tarski recursion. If it is true for a set of formulas, then it will be true for the conjunction, just by taking the supremum of the stabilizing clock values plus one. And the quantifier case is also easy to handle.

Thus, these plays provide a truth predicate for our infinitary language, as desired.

Finally, we need to argue that the interrogator can have no winning strategy. If $\sigma$ is any strategy for the interrogator, then find an ordinal $\theta$ such that $V_\theta$ is closed under that strategy, and have the truth-teller play in accordance with truth in $V_\theta$. This will never violate the Tarski recursion and therefore it will defeat $\sigma$.

So if clopen determinacy holds for this game, then there is a truth predicate for $\mathcal{L}_{\Ord,\omega}(\in,A)$ truth, as desired.\end{proof}

We would like to remark that Gitman and Hamkins have pointed out that there is a flaw in their published proof of~\cite[theorem 9]{GitmanHamkins2016:OpenDeterminacyForClassGames} for the implication of clopen determinacy to \ETR. While that implication is indeed correct as they state, nevertheless it does not suffice for the interrogator to count down merely in the natural numbers, as they had initially claimed. Rather, the interrogator should count down in the order $\omega\cdot\Gamma$, where $\Gamma$ is the length of the iteration, and one can prove by induction in this case that the truth assertions made by the truth-teller about the solution up to any stage $\alpha$ are invariant of the play, provided that the count-down clock is at least $\omega\cdot(\Gamma\restrict\alpha)$. Gitman and Hamkins plan to release an updated version of their paper addressing this matter.

\section{Separating the theorem from other second-order theories}\label{Section.Separating}

In order to situate the class forcing theorem more precisely in the hierarchy of theories between $\GBC$ and $\KM$, let us prove a few theorems that separate $\ETRord$ from other similar principles in the vicinity.

\begin{theorem}
 The theory $\GBC+\ETR_{\Ord\cdot\omega}$ is strictly stronger in consistency strength than $\GBC+\ETRord$.
\end{theorem}

\begin{proof}
Assume $\GBC+\ETR_{\Ord\cdot\omega}$, and fix a class global well-order. By theorem~\ref{Theorem.ETR_Gamma-iff-Gamma-iterated-truth-predicates}, it follows that we may form an $(\Ord\cdot\omega)$-iterated truth predicate $\Tr$ for first-order truth relative to the fixed global well-order (in the language with a predicate for that order). Using this predicate, consider the \GBC\ model having the same sets as $V$, but having as classes only those classes that are definable from a proper initial segment $\Tr\restrict(\Ord\cdot n)$ for some $n<\omega$. This is a \GBC\ model, and furthermore, it satisfies $\ETRord$, because if a class $A$ is definable from $\Tr\restrict\Ord\cdot n$, then we have an $\Ord$-iterated truth predicate definable from $\Tr\restrict\bigl(\Ord\cdot(n+1)\bigr)$. So by statement~\ref{Main.Iterated-truth} in the main theorem, we have $\ETRord$ in this model. But this model cannot have its own $(\Ord\cdot\omega)$-iterated truth predicate, because no such predicate is definable from a proper initial segment of itself. Thus, this is a model of $\ETRord$ without $\ETR_{\Ord\cdot\omega}$. But furthermore, since the entire collection of classes of our constructed model is coded by a single class in our original model, it follows that $\ETR_{\Ord\cdot\omega}$ implies the consistency of $\GBC+\ETRord$, as desired.
\end{proof}

\begin{theorem}
 The theory $\GBC+\ETRord$ has a strictly stronger consistency strength then the theory $\GBC+\ETR_{<\Ord}$, which asserts $\ETR_\alpha$ for every ordinal $\alpha$, provided this latter theory is consistent.
\end{theorem}

\begin{proof}
We may use a similar argument for this. Let $\Tr$ be an $\Ord$-iterated truth predicate for first-order truth, with a fixed class global well-order parameter. Consider the \GBC\ model arising from the sets in $V$ together with any class that is definable from a proper initial segment of the truth predicate $\Tr_\alpha$. This latter \GBC\ model satisfies $\GBC+\ETR_{<\Ord}$ by theorem~\ref{Theorem.ETR_Gamma-iff-Gamma-iterated-truth-predicates}, since one has an $\alpha$-iterated truth predicate with respect to any class that arises.
\end{proof}

One may undertake similar arguments to separate many other levels of the $\ETR_\Gamma$ hierarchy. For example, $\ETR_\omega$ is not provable in \GBC, since one can use it to construct a truth predicate for first-order truth. But $\ETR_\omega$ does not establish $\ETR_{\omega^2}$, since $\ETR_{\omega^2}$ is enough to construct an $\omega$-iterated truth predicate, and then one can take the classes definable from a proper initial segment of it. This will be a \GBC\ model that satisfies $\ETR_\omega$, but not $\ETR_{\omega^2}$.

\section{Final remarks}

We have now proved all the implications necessary to establish the main theorem. Let us make a few final remarks on the topic of the class forcing theorem.

First, we should like to call attention to the fact that only some of the implications in the main theorem hold on a case-by-case basis for the various class forcing notions $\P$ and class parameter $A$. For example, if $\P$ admits a forcing relation for atomic formulas, then by theorems~\ref{Theorem.Atomic-implies-scheme} and~\ref{Theorem.Atomic-to-quantifier-free-forcing-relation} we get a scheme of forcing relations for first-order assertions and also a uniform forcing relation for the quantifier-free infinitary forcing language $\mathcal{L}_{\Ord,0}(\in,V^\P,\dot\Gamma_0,\ldots,\dot\Gamma_m)$, but we do not generally get a uniform forcing relation for the first-order language of set theory or for the stronger infinitary languages. To see this, consider a very nice forcing notion, such as the Easton forcing to control the \GCH\ pattern or even trivial forcing, in a model of \GBC\ having only its definable classes. There is a definable forcing relation for atomic formulas in this model, but there can be no uniform forcing relation for first-order assertions, since from such a relation we could define a truth predicate for the ground model, which is impossible as the model has only definable classes. To get the uniform forcing relation for a class $\P$, what one needs is an instance of $\ETRord$ relative to the forcing relation as a class parameter. In order to complete the cycle of implications in the main theorem, therefore, one applies the statements with stronger and more robust class parameters.

Second, an observant reader might have noticed that we established the equivalences of the main theorem mainly by appealing to an instance of class forcing, the forcing $\FA$, which probably one would rarely want to perform, since this forcing is highly destructive, collapsing the entire universe to become countable, and not just making all sets countable, but rather making $V$ itself countable. Perhaps the reader may wonder if the strength of the class forcing theorem is $\ETRord$ simply because one is including these strange notions of forcing, and that with a more well-behaved collection of class forcing, the principle might be weaker.

In a sense, the objection is correct. To see this, consider the \emph{pretame} class forcing notions, a prominent collection of well-behaved class forcing notions (see~\cite{HolyKrapfSchlicht2017:Characterizations-of-pretameness}). When forcing over the countable transitive models of $\GB$, it turns out that pretameness is equivalent to the preservation of the axioms of $\GB^-$, that is, $\GB$ without the power set axiom. Furthermore, Maurice Stanley has proved that the class forcing theorem holds outright for all pretame class forcing (see~\cite{HolyKrapfSchlicht2017:Characterizations-of-pretameness} for the case of forcing over a countable transitive model of set theory; the argument works generally in \GBC\ and in fact already in $\GB^-$). The class forcing theorem for pretame forcing, consequently, has no extra consistency strength beyond the base theory.

Meanwhile, if one attempts to move beyond pretame class forcing, then we claim that one is immediately again in the realm of the analysis of this article, with highly destructive forcing notions such as $\FA$. The reason is that~\cite[lemmas~2.6,~2.7]{HolyKrapfSchlicht2017:Characterizations-of-pretameness} shows that if a class forcing notion $\P$ is not pretame, then there is some cardinal $\delta$ and a class $\P$-name $\dot F$ forced to be a surjection from $\delta$ to $V$. In other words, any class forcing notion that is not pretame necessarily collapses the entire universe $V$ to a cardinal. Furthermore, making use of the class $\P$-name $\dot F$, an easy reworking of the arguments of section~\ref{Section.Forcing-theorem-implies-truth-for-L_Ord,omega} shows that for every proper class $A$, the forcing $\P$ is dense in a notion of class forcing $\Q_A$ that can be used in place of the forcing $\F_A$ in the argument of theorem~\ref{Theorem.Atomic-forcing-relation-to-truth-predicate}. So if we consider any natural collection of class forcing notions that goes strictly beyond pretame forcing and which includes a forcing notion whenever it includes a dense subclass of that forcing, then the class forcing theorem for this collection will already be as strong as the class forcing theorem for the collection of all class forcing notions.

Let us provide a rough sketch of the reworking idea we mentioned. Suppose that $\P$ is a non-pretame notion of class forcing and that $\dot F$ is a $\P$-name for a surjection from a cardinal $\delta$ to $\Ord$. Given a proper class $A$, we construct the forcing $\Q_A$ from $\P$ in the same manner that we had constructed $\F_A$ from $\Coll(\omega,V)$ in section~\ref{Section.Forcing-theorem-implies-truth-for-L_Ord,omega}. Namely, we add the conditions $e_{n,m}$ as before, as upper bounds for the conditions forcing $\dot F(\check n)\in\dot F(\check m)$, but we do this now for all $n,m<\delta$, and we similarly add the conditions $a_n$ for $n<\delta$ as upper bounds for the conditions forcing $\dot F(\check n)\in A$. In addition, we add a condition $(n\mapsto a)$, for every set $a$ and $n<\delta$, which will be the supremum of all conditions forcing that $\check n$ is the least ordinal with $\dot F(\check n)=\check a$. Since $\dot F$ is only a surjection and not necessarily a bijection, we define $\dot n_a$ so that it is the name of the \emph{least} ordinal $n<\delta$ that will get mapped to $a$. Under this setup, we may carry out the analogue of the proof of theorem~\ref{Theorem.Atomic-forcing-relation-to-truth-predicate}, replacing $\F_A$ by $\Q_A$ and $\omega$ by $\delta$, respectively. One minor change consists of isolating a least, rather than an arbitrary, ordinal $n<\delta$ for which $q$ forces $\lnot\varphi^*(\check n,\dot n_a)$ in the quantifier case in the proof of lemma~\ref{Lemma.p-forces-iff-one-forces}; this ensures that we can choose $q$ and find some set $b$ so that $q$ forces $\check n=\dot n_b$, as in the corresponding step in the proof of lemma~\ref{Lemma.p-forces-iff-one-forces}.

In summary, the class forcing theorem for pretame forcing is provable in \GBC, but for any sufficiently robust collection of class forcing notions going beyond the pretame forcing, the class forcing theorem will have the full strength of $\ETRord$ and all the other statements of the main theorem.

\bibliographystyle{alpha}
\bibliography{HamkinsBiblio,MathBiblio,WebPosts}
\end{document}

%% file: MathMacrosJDH.tex
\newtheorem{theorem}{Theorem}

\newtheorem*{maintheorem*}{Main Theorem}
\newtheorem*{maintheorems*}{Main Theorems}

\newtheorem*{corollary*}{Corollary}
\newtheorem*{corollaries*}{Corollaries}
\newtheorem{sublemma}{Lemma}[theorem]
\newtheorem{lemma}[theorem]{Lemma}

\newtheorem*{question*}{Question}

\newtheorem*{questions*}{Questions}
\newtheorem*{mainquestion*}{Main Question} % without numbering
\newtheorem*{openquestion*}{Open Question} % without numbering

\newtheorem{definition}[theorem]{Definition}

\newcommand{\QED}{\end{proof}}

\def\proclaim[#1]{{\bf #1}}
\def\BF#1.{{\bf #1.}}
\newcommand{\Godel}{G\"odel}

\newcommand{\B}{{\mathbb B}}

\newcommand{\F}{{\mathbb F}}

\renewcommand{\P}{{\mathbb P}}
\newcommand{\Q}{{\mathbb Q}}

\newcommand{\T}{{\mathbb T}}
\newcommand{\one}{\mathbbm{1}} % requires \usepackage{bbm}
\newcommand{\from}{\mathbin{\vbox{\baselineskip=2pt\lineskiplimit=0pt
                         \hbox{.}\hbox{.}\hbox{.}}}}

\newcommand{\of}{\subseteq}

\newcommand{\set}[1]{\{\,{#1}\,\}}

\newcommand{\rank}{\mathop{\rm rank}}

\newcommand{\Coll}{\mathop{\rm Coll}}

\newcommand{\Con}{\mathop{{\rm Con}}}

\newcommand{\restrict}{\upharpoonright} % uses amssymb
\newcommand{\satisfies}{\models}
\newcommand{\forces}{\Vdash}

\newcommand{\concat}{\mathbin{{}^\smallfrown}}

\newcommand{\union}{\cup}

\newcommand{\intersect}{\cap}

\newcommand{\smalllt}{\mathrel{\mathchoice{\raise2pt\hbox{$\scriptstyle<$}}{\raise1pt\hbox{$\scriptstyle<$}}{\raise0pt\hbox{$\scriptscriptstyle<$}}{\scriptscriptstyle<}}}
\newcommand{\smallleq}{\mathrel{\mathchoice{\raise2pt\hbox{$\scriptstyle\leq$}}{\raise1pt\hbox{$\scriptstyle\leq$}}{\raise1pt\hbox{$\scriptscriptstyle\leq$}}{\scriptscriptstyle\leq}}}

\newcommand{\boolval}[1]{\mathopen{\lbrack\!\lbrack}\,#1\,\mathclose{\rbrack\!\rbrack}}
\def\[#1]{\boolval{#1}}
\newbox\gnBoxA
\newdimen\gnCornerHgt
\setbox\gnBoxA=\hbox{\tiny$\ulcorner$}
\global\gnCornerHgt=\ht\gnBoxA
\newdimen\gnArgHgt
\def\gcode #1{%
\setbox\gnBoxA=\hbox{$#1$}%
\gnArgHgt=\ht\gnBoxA%
\ifnum     \gnArgHgt<\gnCornerHgt \gnArgHgt=0pt%
\else \advance \gnArgHgt by -\gnCornerHgt%
\fi \raise\gnArgHgt\hbox{\tiny$\ulcorner$} \box\gnBoxA %
\raise\gnArgHgt\hbox{\tiny$\urcorner$}}
\newcommand{\UnderTilde}[1]{{\setbox1=\hbox{$#1$}\baselineskip=0pt\vtop{\hbox{$#1$}\hbox to\wd1{\hfil$\sim$\hfil}}}{}}
\newcommand{\Undertilde}[1]{{\setbox1=\hbox{$#1$}\baselineskip=0pt\vtop{\hbox{$#1$}\hbox to\wd1{\hfil$\scriptstyle\sim$\hfil}}}{}}
\newcommand{\undertilde}[1]{{\setbox1=\hbox{$#1$}\baselineskip=0pt\vtop{\hbox{$#1$}\hbox to\wd1{\hfil$\scriptscriptstyle\sim$\hfil}}}{}}
\newcommand{\UnderdTilde}[1]{{\setbox1=\hbox{$#1$}\baselineskip=0pt\vtop{\hbox{$#1$}\hbox to\wd1{\hfil$\approx$\hfil}}}{}}
\newcommand{\Underdtilde}[1]{{\setbox1=\hbox{$#1$}\baselineskip=0pt\vtop{\hbox{$#1$}\hbox to\wd1{\hfil\scriptsize$\approx$\hfil}}}{}}

\renewcommand{\iff}{\mathrel{\leftrightarrow}}
\newcommand{\Iff}{\mathrel{\Longleftrightarrow}}

\def\<#1>{\left\langle#1\right\rangle}

\newcommand{\Ord}{\mathord{{\rm Ord}}}
\newcommand{\ETR}{{\rm ETR}}

\newcommand{\ZFC}{{\rm ZFC}}
\newcommand{\ZF}{{\rm ZF}}

\newcommand{\KM}{{\rm KM}}
\newcommand{\GB}{{\rm GB}}
\newcommand{\GBC}{{\rm GBC}}

\newcommand{\GCH}{{\rm GCH}}

\newcommand{\AC}{{\rm AC}}

\newcommand{\cell}[1]{\boxit{\hbox to 17pt{\strut\hfil$#1$\hfil}}}
\newcommand{\head}[2]{\lower2pt\vbox{\hbox{\strut\footnotesize\it\hskip3pt#2}\boxit{\cell#1}}}
\newcommand{\boxit}[1]{\setbox4=\hbox{\kern2pt#1\kern2pt}\hbox{\vrule\vbox{\hrule\kern2pt\box4\kern2pt\hrule}\vrule}}
\newcommand{\Col}[3]{\hbox{\vbox{\baselineskip=0pt\parskip=0pt\cell#1\cell#2\cell#3}}}
\newcommand{\tapenames}{\raise 5pt\vbox to .7in{\hbox to .8in{\it\hfill input: \strut}\vfill\hbox to
.8in{\it\hfill scratch: \strut}\vfill\hbox to .8in{\it\hfill output: \strut}}}
\newcommand{\Head}[4]{\lower2pt\vbox{\hbox to25pt{\strut\footnotesize\it\hfill#4\hfill}\boxit{\Col#1#2#3}}}
\newcommand{\Dots}{\raise 5pt\vbox to .7in{\hbox{\ $\cdots$\strut}\vfill\hbox{\ $\cdots$\strut}\vfill\hbox{\
$\cdots$\strut}}}

%% file: Forcing_theorem.bbl
\newcommand{\etalchar}[1]{$^{#1}$}
\begin{thebibliography}{HKL{\etalchar{+}}16}

\bibitem[Fuj12]{Fujimoto2012:Classes-and-truths-in-set-theory}
Kentaro Fujimoto.
\newblock Classes and truths in set theory.
\newblock {\em Ann. Pure Appl. Log.}, 163(11):1484--1523, 2012.

\bibitem[GH16]{GitmanHamkins2016:OpenDeterminacyForClassGames}
Victoria Gitman and Joel~David Hamkins.
\newblock Open determinacy for class games.
\newblock In Andr\'es~E. Caicedo, James Cummings, Peter Koellner, and Paul
  Larson, editors, {\em Foundations of Mathematics, Logic at Harvard, Essays in
  Honor of Hugh Woodin's 60th Birthday}, Contemporary Mathematics. American
  Mathematical Society, 2016. 

\bibitem[HKL{\etalchar{+}}16]{HolyKrapfLuckeNjegomirSchlicht2016:Class-forcing}
Peter Holy, Regula Krapf, Philipp {L\"ucke}, Ana Njegomir, and Philipp
  Schlicht.
\newblock Class forcing, the forcing theorem and {Boolean} completions.
\newblock {\em J. Symb. Log.}, 81(4):1500--1530, 2016.

\bibitem[HKS17]{HolyKrapfSchlicht2017:Characterizations-of-pretameness}
Peter Holy, Regula Krapf, and Philipp Schlicht.
\newblock Characterizations of pretameness and the {O}rd-cc.
\newblock 2017.
\newblock {\em Ann. Pure Appl. Log.}, 169(8):775--802, 2018. 

\bibitem[HL16]{HamkinsLeahy2016:AlgebraicityAndImplicitDefinabilityInSetTheory}
Joel~David Hamkins and Cole Leahy.
\newblock Algebraicity and implicit definability in set theory.
\newblock {\em Notre Dame J. Formal Logic}, 57(3):431--439, 2016.

\bibitem[HS]{HamkinsSeabold:BooleanUltrapowers}
Joel~David Hamkins and Daniel Seabold.
\newblock Well-founded {Boolean} ultrapowers as large cardinal embeddings. 
\newblock Preprint, available at https://arxiv.org/abs/1206.6075. 


\bibitem[HY]{HamkinsYang:SatisfactionIsNotAbsolute}
Joel~David Hamkins and Ruizhi Yang.
\newblock Satisfaction is not absolute.
\newblock {\em To appear in Rev. Symb. Log.} 
\newblock Available at https://arxiv.org/abs/1312.0670. 


\bibitem[Kra74]{Krajewski1974:MutuallyInconsistentSatisfactionClasses}
S.~Krajewski.
\newblock Mutually inconsistent satisfaction classes.
\newblock {\em Bull. Acad. Polon. Sci. S\'er. Sci. Math. Astronom. Phys.},
  22:983--987, 1974.

\bibitem[Kra17]{Krapf2017:Dissertation}
Regula Krapf.
\newblock {\em Class forcing and second-order arithmetic}.
\newblock PhD thesis, University of Bonn, 2017.

\bibitem[Sat14]{Sato2014:Relative-predicativity-and-dependent-recursion-in-second-order-set-theory-and-higher-order-theories}
Kentaro Sato.
\newblock Relative predicativity and dependent recursion in second-order set
  theory and higher-order theories.
\newblock {\em J. Symb. Log.}, 79(3):712--732, 2014.

\bibitem[Wil19]{williams-min-km}
Kameryn J. Williams.
\newblock Minimum models of second-order set theories.
\newblock {\em J. Symb. Log.}, 8(2):589--620, 2019.

\end{thebibliography}
